\documentclass[a4paper,11pt]{article}

\usepackage{mathrsfs}		
\usepackage{amsfonts}		
\usepackage{amsthm}		
\usepackage{enumerate}		
\usepackage{amssymb}		
\usepackage{amsmath}		
\usepackage{vmargin}		
\usepackage{graphicx}		
\usepackage{stmaryrd}		
\usepackage{hyperref}		
\usepackage{tikz}		
\usepackage{subfigure}		
\usepackage{enumitem}

\usetikzlibrary{patterns}	

\hypersetup{
    colorlinks=true,		
    linkcolor=black,		
    citecolor=black,		
    filecolor=black,		
    urlcolor=black		
}

\newcommand{\scr}[1]{\ensuremath {\mathscr{#1}}}
\newcommand{\mat}[1]{\ensuremath {\mathbb{#1}}}
\newcommand{\scal}[2]{\ensuremath \langle {#1}, {#2} \rangle}
\newcommand{\gen}[1]{\ensuremath \left\langle {#1} \right\rangle}
\newcommand{\call}[1]{\ensuremath {\mathcal{#1}}}
\renewcommand{\bar}[1]{\ensuremath \overline{#1}}
\renewcommand{\hat}[1]{\ensuremath \widehat{#1}}
\renewcommand{\tilde}[1]{\ensuremath \widetilde{#1}}

\newcommand{\sx} {\ensuremath {\left(}}
\newcommand{\dx} {\ensuremath {\right)}}
\newcommand{\nbc}{\ensuremath {\operatorname{nbc}}}
\newcommand{\rk} {\ensuremath {\mbox{rk}}\,}
\newcommand{\Hom} {\ensuremath {\mbox{Hom}}}
\newcommand{\sopra}[1] {\ensuremath {{#1}^\upharpoonright}}

\newcommand{\acicat}{\boldsymbol{AC}}
\newcommand{\cat}{\boldsymbol{Cat}}
\newcommand{\Sal}{\operatorname{Sal}}
\newcommand{\colim}{\operatorname{colim}}
\newcommand{\codim}{\operatorname{codim}}
\newcommand{\ob}{\operatorname{Ob}}
\newcommand{\SX}{\llbracket}
\newcommand{\DX}{\rrbracket}

\newcommand{\Obj}{\operatorname{Obj}}
\newcommand{\Mor}{\operatorname{Mor}}
\newcommand{\match}{\mathfrak{M}}
\newcommand{\id}{\operatorname{id}}

\theoremstyle{definition}
\newtheorem{teo}{Theorem}
\newtheorem{definition}[teo]{Definition}
\newtheorem{assu}[teo]{Assumption}
\newtheorem{prop}[teo]{Proposition}
\newtheorem{lemma}[teo]{Lemma}
\newtheorem{cor}[teo]{Corollary}
\newtheorem{lem}[teo]{Lemma}

\theoremstyle{remark}
\newtheorem{oss}{Remark}
\newtheorem{es}[teo]{Example}

\setpapersize{A4}
\setmarginsrb{35mm}{25mm}{35mm}{25mm}%
            {0mm}{25mm}{0mm}{25mm}

\title{Minimality of toric arrangements}
\author{Giacomo d'Antonio\footnote{Email: \texttt{dantonio@math.uni-bremen.de}} 
~and Emanuele Delucchi\footnote{Email: \texttt{delucchi@math.uni-bremen.de}}
}

\begin{document}

\maketitle
\begin{abstract}
We prove that the complement of a toric arrangement has the homotopy
type of a minimal CW complex. As a corollary we obtain that the
integer cohomology of these spaces is torsion free.

We use Discrete Morse Theory, providing a sequence of cellular
collapses that leads to a minimal complex. 
\end{abstract}

\setcounter{tocdepth}{1}

\section*{Introduction}

A {\em toric arrangement} is a finite family
\begin{displaymath}
  \scr A = \{K_1,\ldots ,K_n\} 
\end{displaymath}
of special subtori of the complex
torus $(\mathbb C^*)^d$ (more precisely the $K_i$ are level sets of
characters, see \S \ref{sec:toric_introduction}). 
Given a complexified toric arrangement $\scr A$ (see Definition \ref{def:toric_abstract}) we consider the space
\begin{displaymath}
  M(\scr A):=(\mathbb C^*)^d\setminus \bigcup \scr A
\end{displaymath}
and prove that 
\begin{itemize}
\item[(a)] the space $M (\scr A)$ is {\em minimal} in the sense of \cite{MR2018927}, i.e., it has the
  homotopy type of a CW complex with exactly $\beta_k =
  \mbox{rk}\,H^k(M(\scr A); \mat Z)$ cells in dimension $k$, for every
  $k \in \mat N$, hence
\item[(b)] the space $M(\scr A)$ is {\em torsion-free}, that is, the homology and cohomology modules $H_k(X; \mat Z)$, $H^k(X; \mat Z)$ are
torsion free for every $k\in \mathbb N$. 
\end{itemize}

The study of toric arrangements  
 experienced a
fresh impulse  from recent work of De Concini, Procesi and Vergne
\cite{DPV,deconcini2010topics}, in which toric arrangements emerge as a link
between 
partition functions and 
 box splines.


In their book \cite{deconcini2010topics}, De Concini and Procesi
 emphasize some similarities between  toric arrangements
 and the well-established theory of arrangements of affine hyperplanes.
%
%
%
%
%
The present work provides substantial new evidence in this sense. 


\subsubsection*{Background}
\paragraph{Combinatorics.} The combinatorial framework for the theory of arrangements of
hyperplanes is widely considered to be given by matroid theory, a well-established branch of
combinatorics that has proved very useful in this context ever since
the seminal work of Zaslavsky \cite{zaslavsky1975facing}.

The combinatorial study of toric arrangements has quite recent roots,
and is still in search of a full-fledged pertaining theory.
From an enumerative point of view, the arithmetic Tutte polynomial introduced by Moci in  
\cite{mocitutte2009} summarizes previous results by Ehrenborg, Readdy
and Slone \cite{ehrenborg2009affine} and of De Concini and Procesi
\cite{deconcini2010topics}. This initiated the quest for a variation on the
concept of matroid that would suit the `toric' setting and lead
D'Adderio and Moci \cite{mocidadd} to suggest a theory of arithmetic
matroids as a ``combinatorialization'' of the essential algebraic data of
toric arrangements. Arithmetic matroids in fact encode - but, as yet, do not appear to
characterize - some of the crucial combinatorial data of toric
arrangements, for example the poset of layers (Definition \ref{def:layerposet}). 
In this context, our work can be seen as exploration of the properties
that would be required from a (still lacking) notion of a `toric oriented matroid'.

\paragraph{Topology}
An important result in the
theory of arrangements of hyperplanes was established by Brieskorn
\cite{brieskorn1971gt}, who proved that the integer cohomology of the
complement of an arrangement of complex hyperplanes is
torsion-free. This allowed Orlik and Solomon to compute the integer
cohomology algebra via the deRham complex \cite{orlik1980cat}. Minimality of complements of complex hyperplane
arrangements was proven much later by Randell in
\cite{MR1900880} and independently by Dimca and Papadima in \cite{MR2018927}, with Morse theoretic arguments. The explicit
construction of such a minimal complex was studied by Yoshinaga
\cite{yoshinaga2007}, Salvetti and Settepanella \cite{MR2350466} and
the second author \cite{delucchi}.

The present paper
completes a similar circle of ideas for toric arrangements.

To our knowledge, the first result about the topology of toric
arrangements was obtained by Looijenga \cite{looij} who deduced the Betti
numbers of $M(\scr A)$ from a spectral sequence computation.
De Concini and Procesi in \cite{de2005geometry} explicitely expressed the generators
of the cohomology modules over $\mathbb C$ in terms of local no broken
circuit sets and, for the special case of
totally unimodular arrangements, were able to compute the
cohomological algebra structure. 
A presentation of the fundamental group $\pi_1(M(\scr A))$ of complexified toric
arrangements was computed by the authors in
\cite{dantoniodelucchi}, based on a combinatorially defined polyhedral
complex carrying the homotopy type of the complement $M(\scr A)$,
called {\em toric Salvetti complex}. This polyhedral complex is given
as the nerve of an acyclic category\footnote{For our use of the term
  `acyclic category' see Remark \ref{oss:acycliccat}} and was introduced by the
authors in \cite{dantoniodelucchi}, generalizing to arbitrary
complexified toric arrangements the complex defined by Moci and
Settepanella in \cite{mocisettepanella}. 
Recently, Davis and Settepanella \cite{davis_settepanella} published vanishing
results for cohomology of toric arrangements with coefficients in some
particular local systems.

\subsubsection*{Outline}

Here we prove minimality by exhibiting, for a given complexified toric arrangement $\scr A$, a
minimal CW-complex that is homotopy equivalent to $M(\scr A)$. This
complex is obtained from the
toric Salvetti complex after a sequence of cellular collapses indexed
by a discrete Morse function. The construction of the discrete Morse
function relies on a stratification of the toric Salvetti complex
where strata are counted by `local no-broken-circuit sets'
(Definition \ref{def:lnbc}), which are known to control the Poincar\'e
polynomial of $M(\scr A)$ by \cite{de2005geometry}.

The (topological) boundary maps of the minimal
complex can be recovered in principle from the Discrete Morse
data. The explicit computation of such boundary maps is in general difficult even in the case
of hyperplane arrangements, where explicit computations are known only
in dimension $2$ either by following the discrete Morse gradient
\cite{gaiffisalvetti,GMS} or by exploting braid monodromy \cite{Hi93,Sa88,Sa882}.
We leave the explicit computation of the boundary maps for our toric
complex as a future direction of research.

As an application of our methods, 
in the last section we 
describe a 
construction of the minimal complex for complexified affine
arrangements of hyperplanes that uses only the intrinsic
combinatorics of the arrangement (i.e. its oriented matroid), as an alternative to the method of
\cite{MR2350466}.





We close our introduction with a detailed outline of the paper.

\begin{itemize}
\item 
  We begin with Section \ref{sec:arrangements}, where we review some known
  facts about the combinatorics and the topology of hyperplane
  arrangements and we prove some preparatory results about linear
  extensions of posets of regions of real arrangements.
\item In Section \ref{sec:toric} we give a short introduction to toric
  arrangements and we collect some results from the 
  literature on which our work is built.
\item Section \ref{sec:DMT} breaks the flow of material directly
  related to toric arrangements in order to develop Discrete Morse
  Theory for acyclic categories, generalizing the existing theory for
  posets.
\item We approach the core of our work with Section
  \ref{section:strata}, where we introduce a stratification and a related
  decomposition of the toric
  Salvetti complex (Definition \ref{def:N}).
\item In order to understand the structure of the pieces of the
  decomposition of the toric Salvetti complex we need to patch
  together `local' combinatorial data, which come from the theory of
  arrangements of hyperplanes. We do this in Section
  \ref{sec:topstrata} using diagrams of acyclic categories. 
\item Our work culminates with Section \ref{sec:minimality}. The
  keystone is Proposition \ref{prop.minreale}, where we prove the
  existence of perfect acyclic matchings for the face categories of 
  subdivisions of the compact torus given by toric arrangements. With
  this, we can apply the Patchwork Lemma of Discrete Morse Theory (in
  its version for acyclic categories) to
  our decomposition of the toric Salvetti complex to get an acyclic
  matching of the whole complex. This matching can be shown to be perfect and
  thus prescribes a series of cellular collapses leading to a minimal
  model for the complement of the toric arrangement.
\item As a further application of our methods, in Section \ref{sec:affine} we show that our methods can be
  used to construct a minimal complex for the complement of (finite)
  complexified arrangements of hyperplanes. 
\end{itemize}
\section{Arrangements of hyperplanes}
\label{sec:arrangements}


The theory of hyperplane arrangements is an important ingredient in
our treatment of toric arrangements.
In order to set the stage for the subsequent considerations, we
therefore introduce the language and recall some relevant results about hyperplane arrangements. A standard reference for a comprehensive introduction to the
subject is \cite{orlik1992ah}.

\subsection{Generalities}

Through this section let $V$ be a finite dimensional vector space over
a field $\mathbb K$.

An {\em affine hyperplane} $H$ in $V$ is
the level set of a linear functional on $V$. That is, there is
$\alpha\in V^*$ and $a\in \mathbb K$ such that $H=\{v\in V \mid
\alpha(v)=a\}$.
%
  A set of hyperplanes is called {\em dependent} or {\em independent} according to whether the corresponding
  set of elements of $V^*$ is dependent or not.

\begin{definition}
  An \emph{arrangement of hyperplanes} in $V$ is a collection $\scr A$
  of affine hyperplanes in $V$.
\end{definition}

An hyperplane arrangement $\scr A$ is called
\emph{central} if every hyperplane $H \in \scr A$ is a linear subspace
of $V$; 
\emph{finite} if $\scr A$ is finite;
\emph{locally finite} if for every $p \in V$ the set $\{H \in \scr A\mid p \in H\}$ is
  finite;
\emph{real} (or \emph{complex}, or \emph{rational}) if $V$ is a 
  real (or complex, or rational) vector space.

When we will need to define a total order on the elements of a finite arrangement $\scr A$,
we will do this by simply indexing the elements of $\scr A$, as
$\scr A = \{H_1, \dots, H_n\}$.

Much of the theory of hyperplane arrangements is devoted to the study
of the \emph{complement} of an arrangement $\scr A$. That is, the space
\begin{displaymath}
  M(\scr A) := V \backslash  \bigcup \scr A.  
\end{displaymath}

\begin{definition}\label{def:IL}
  Let $\scr A$ be an hyperplane arrangement, the \emph{intersection
    poset} of $\scr A$ is the set
  \begin{displaymath}
    \call L(\scr A) := \big\{\bigcap \scr K \,\big\vert\, \scr K \subseteq \scr A \big\} \backslash \big\{\emptyset \big\}
  \end{displaymath}
  of all nonempty intersections of elements of $\scr A$, ordered by
  \emph{reverse} inclusion - i.e., for $X,Y\in \call L(\scr A)$,  $X\geq Y$ if $X\subseteq Y$.  
\end{definition}
Notice that the whole space $V$ is an
element of $\call L(\scr A)$ (corresponding to the empty
intersection), whereas the empty set is not. The intersection poset is
a {meet-semilattice} and for central hyperplane arrangements is a {lattice}. Then, we speak of {\em  intersection lattice} of $\scr A$.

\subsubsection{Deletion and restriction}
Consider a hyperplane arrangement $\scr A$ in the vector space $V$
and an intersection $X \in \call L(\scr A)$. We associate to $X$ two new arrangements:
\begin{displaymath}
  \scr A_X = \{ H \in \scr A \mid X \subseteq H\},\quad
  \scr A^X = \{ H \cap X \mid H \in \scr A \backslash \scr A_X\}.
\end{displaymath}
Notice that $\scr A_X$ is an arrangement in $V$, while $\scr A^X$ is
an arrangement on $X$.

\begin{oss}
  If a total ordering $\scr A = \{H_1,\ldots, H_n\}$ is defined, then it
  is clearly inherited by $\scr A_X$ for every $X\in \call L(\scr
  A)$. On the elements of $\scr A^X$ a total ordering is induced as
  follows. For $L\in \scr A^X$ define 
  \begin{equation}
    \label{defXL}
    _XL:=\min \{H\in \scr A \mid L\subseteq H\}.
  \end{equation}
  Then, order $\scr A^X:=\{L_1,\ldots , L_m\}$ so that, for all $1\leq
  i<j\leq m$, $_XL_i < \, _XL_J$ in $\scr A$.
\end{oss}

\subsubsection{No Broken Circuit sets}
 In this section let $\scr A$ be a {central} hyperplane
 arrangement and fix an arbitrary total ordering of $\scr A$.

\begin{definition} 
  A \emph{circuit} is a minimal dependent subset $C \subseteq \scr
  A$. A \emph{broken circuit} is a subset of the form
  \begin{displaymath}
    C \backslash \{\mbox{min}\,C\} \subseteq \scr A
  \end{displaymath}
  obtained from a circuit removing its least element.
  A \emph{no broken circuit set} (or, for short, an {\em nbc} set) is a subset $N \subseteq \scr A$ which
  does not contain any broken circuit.
\end{definition}

\begin{oss}
  An equivalent definition of nbc set is the following.
  A subset $N = \{H_{i_1}, \dots, H_{i_k}\} \subseteq \scr A$ with
  $i_1 \leq \cdots \leq i_k$ is a \emph{no broken circuit set} if
  it is independent and there is no $j < i_1$ such that
  $N \cup \{H_j\}$ is dependent.
\end{oss}

\begin{definition}
  We will write $\nbc(\scr A)$ for the set of no broken circuit sets
  of $\scr A$
  and $\nbc_k(\scr A) = \{ N \in \nbc(\scr A) \mid |N| = k \}$ for the set
  of all no broken circuit sets of cardinality $k$.
\end{definition}

\subsection{Real arrangements}
\label{sec:realarrangements}

In this section we consider the case where $\scr A$ is an arrangement
of hyperplanes in $\mat R^d$ in order to set up some notation and use
the real structure to gain some deeper understanding in the
combinatorics of no broken circuit sets.
 
It is not too difficult to verify that the complement $M(\scr A)$
 consists of several contractible connected components.
These are called \emph{chambers} of $\scr A$. We write $\call T(\scr A)$ for
the set of all chambers of $\scr A$.

\begin{definition}\label{def:faces}
  Let $\scr A$ a real arrangement, the set of \emph{faces} of $\scr A$ is
  \begin{displaymath}
    \call F(\scr A) := \{\operatorname{relint}(\bar C \cap X) \mid C \in \call T(\scr A),\, X \in \call L(\scr A)\}.
  \end{displaymath}
  We partially order this set by setting $F\leq G$ if $ F \subseteq
  \bar G$ and call then $\call F(\scr A)$ the \emph{face poset} of $\scr A$.

\end{definition}

\begin{oss}
  A face $F\in \call F(\scr A)$ is an open subset of $\bigcap
  \{H\in \scr A \mid F\subseteq H\}$. By $\overline{F}$ we mean the topological closure of $F$
  in $\mathbb R^d$.
\end{oss}

\begin{oss}\label{rem:Rsubarr}
  Given $F\in \call F(\scr A)$ define the subarrangement $\scr
  A_F:=\{H\in \scr A \mid F\subseteq H\}$. We have a natural poset
  isomorphism $\call F(\scr A_F)\simeq \call F(\scr A)_{\geq
    F}$. Therefore, in the following we
  will identify these two posets.
\end{oss}

One of the main enumerative questions about arrangements of
hyperplanes in real space asks for the number of chambers of a given
hyperplane arrangement. The answer is very elegant and somehow surprising.
\begin{teo}[Zaslavsky \cite{zaslavsky1975facing}]
  Given a real hyperplane arrangement $\scr A$, 
  \begin{displaymath}
    |\call T(\scr A)| = |\nbc(\scr A)|.
  \end{displaymath}
\end{teo}


\subsubsection{Taking sides}
If $\scr A$ is an arrangement in a real space $V$, then every hyperplane $H$ is
the locus where a linear form $\alpha_H\in V^*$ takes the value $a_H$. This way we can associate to each $H\in \scr A$, its
\emph{positive} and \emph{negative halfspace}:
\begin{displaymath}
  H^+ = \{x \in  V \mid \alpha_H(x) > a_H\},\quad\quad
  H^- = \{x \in  V \mid \alpha_H(x) < a_H\}.
\end{displaymath}


\begin{definition}\label{def:sign vectors}
  Consider a complexified locally finite arrangement $\scr A$ 
with any choice of `sides'
$H^+$ and $H^-$ for every $H\in \scr A$. 
   The {\em sign vector} of a face $F \in \call F(\scr A)$ is the function
  $\gamma_F:\scr A \to \{-, 0 +\}$ defined as:
  \begin{displaymath}
    \gamma_F(H) :=
    \left\{
      \begin{array}{ll}
        + & \mbox{ if } F \subseteq H^+,\\
        0 & \mbox{ if } F \subseteq H,\\
        - & \mbox{ if } F \subseteq H^-.
      \end{array}
    \right.
  \end{displaymath}
  When we will need to specify the arrangement $\scr A$ to
  which the sign vector refers, we will write $\gamma[\scr A]_F(H)$
  for $\gamma_F(H)$.
\end{definition}
\begin{oss}
  \label{rem:FSV} 
  The poset $\call F(\scr A)$ is isomorphic to the set $\{\gamma_F
  \mid F\in \call F(\scr A)\}$ with partial order given by $\gamma_F
  \leq \gamma_G$ if $\gamma _F(H) = \gamma_G(H)$ whenever
  $\gamma_G(H)\neq 0$ (see e.g. \cite{BLSWZ}). 
\end{oss}


\begin{definition}\label{def_sep}
  Let $C_1$ and $C_2 \in \call T(\scr A)$ be chambers of a real
  arrangement, and let  $B\in \call T(\scr A)$ be a distinguished chamber. We will write
  \begin{displaymath}
    S(C_1,C_2) := \{H \in \scr A \mid  \gamma_{C_1}(H) \neq \gamma_{C_2}(H)\}
  \end{displaymath}
  for the set of hyperplanes of $\scr A$ which separate $C_1$ and
  $C_2$.\\
  For all $C_1, C_2 \in \call T(\scr A) $ write
  \begin{displaymath}
    C_1 \leq C_2 \iff S(C_1, B) \subseteq S(C_2, B).
  \end{displaymath}
  This turns $\call T(\scr A)$ into a poset $\call T(\scr A)_B$, the
  {\em poset of regions} of the arrangement $\scr A$ with base chamber
  $B$.
\end{definition}

\begin{oss}
  Let $\scr A_0$ be a real arrangement and $B\in \call T (\scr
  A_0)$. Given a subarrangement $\scr A_1 \subseteq \scr A_0$, for
  every chamber  $C\in \call T(\scr A_0)$ there is a unique chamber
  $\hat{C}\in \call T(\scr A_1)$ with $C \subseteq \hat{C}$. The correspondence
  $C\mapsto \hat C$ defines a surjective map
  \begin{displaymath}
    \sigma_{\scr A_1}: \call T(\scr A_0)_B \to \call T(\scr
    A_1)_{\hat{B}}
  \end{displaymath}
  such that $C\leq C'$ implies $\sigma_{\scr A_1}(C)\leq \sigma_{\scr
    A_1}(C')$ for all $C,C'\in \call T(\scr A_0)$.
\end{oss}


\begin{definition}\label{def:mu}
    Let $\scr A_0$ be a real arrangement and let $\succ_0$ denote any
    total ordering of 
$\call T (\scr A_0)$. Consider a subarrangement $\scr
  A_1\subseteq \scr A_0$. The section
  \begin{displaymath}
    \mu[\scr A_1, \scr A_0]: \call T(\scr A_1) \to \call T(\scr A_0),\quad
    C\mapsto \min_{\succ_0} \{K\in \call T(\scr A_0) \mid K\subseteq C\} 
  \end{displaymath}
  of $\sigma_{\scr A_1}$ defines a total ordering $\succ_{0,1}$ on $\call
  T (\scr A_1)$ by
  \begin{displaymath}
    C \succ_{0,1} D \iff \mu[\scr A_1,\scr A_0](C) \succ_0 \mu[\scr A_1,\scr A_0](D)
  \end{displaymath}
  that we call {\em induced by $\succ_0$}.
\end{definition}

\begin{lemma}\label{lemma_mu}
  Consider real arrangements $\scr A_2\subseteq \scr A_1 \subseteq
  \scr A_0$, a given total ordering $\succ_0$ of $\call T(\scr A_0)$ and the
  induced total ordering $\succ_{0,1}$ of $\call T(\scr A_1)$. Then 
  \begin{displaymath}
    \mu[\scr A_1,\scr A_0]\circ \mu[\scr A_2,\scr A_1]=\mu[\scr A_2,\scr A_0].
  \end{displaymath}
\end{lemma}

\begin{proof}
  Take any $C\in \call T(\scr A_2)$ and define 
  \begin{displaymath}\begin{array}{ll}
   C_0:=\mu[\scr A_2,\scr A_0](C);& C_1:=\sigma_{\scr A_1}(C_0),\textrm{
     so } \mu[\scr A_1,\scr A_0](C_1)=C_0;\\
   C_2:=\mu[\scr A_2,\scr A_1](C);& C_3:=\mu[\scr A_1,\scr
   A_0](C_2).
   \end{array}
  \end{displaymath}
  we have to show that $C_0=C_3$. \\
  First, notice that $C_0\preceq_0 C_3$ because $C_3\subseteq
  C_2\subseteq C$. For the reverse inequality notice that we have  $C_1,C_2\subseteq C$, which implies
  $C_2\preceq_{0,1} C_1$ and so, by definition of the induced ordering,
  $C_3=\mu[\scr A_1,\scr A_0](C_2) \preceq_0 \mu[\scr A_1,\scr A_0](C_1)=C_0$.
\end{proof}

\begin{prop}\label{prop_indext} Let a base chamber $B$ of $\scr A_0$
  be chosen.
  If $\succ_0$ is a linear extension of $\call T(\scr A_0)_B$, then
  $\succ_{0,1}$ is a linear extension of $\call T(\scr A_1)_{\hat{B}}$.
\end{prop}
\begin{proof}
  We have to prove that for all $C,D\in \call T(\scr A_1)$, 
  $C\leq D$ in $\call T(\scr A_1)_{\hat{B}}$ implies
  $C\preceq_{0,1}D$, i.e., $\mu[\scr A_0,\scr
  A_1](C)\preceq_0 \mu[\scr A_0,\scr A_1](D)$.

  We argue by induction on $k:=\vert \scr A_0 \setminus \scr A_1
  \vert$, the claim being evident when $k=0$. Suppose then that $k>0$,
  choose $H\in \scr A_0\setminus \scr A_1$ and set $\scr A_0':=\scr
  A_0\setminus \{H\}$. By induction
  hypothesis we have
  \begin{displaymath}
    \mu[\scr A_0',\scr A_1](C)\preceq_0' \mu[\scr A_0',\scr A_1](D),
  \end{displaymath}
  which by definition means 
  \begin{displaymath}
    \mu[\scr A_0,\scr A_0'](\mu[\scr A_0',\scr A_1](C))\preceq_0 \mu[\scr A_0,\scr A_0'](\mu[\scr A_0',\scr A_1](D))
  \end{displaymath}
and thus, via Lemma \ref{lemma_mu},
  $\mu[\scr A_0,\scr A_1](C) \preceq_0 \mu[\scr A_0,\scr A_1](D).$
\end{proof}

\subsection{Complex(ified) arrangements}
\label{sec:complexified}

We turn to the case of complex hyperplane arrangements, where the
space $M(\scr A)$ has subtler topology. For the sake of concision here
we deliberately disregard the chronological order in which the
relevant theorems were proved, and start with the minimality result.

\begin{definition}
  Let $X$ be a topological space. For $j\geq 0$, the {\em $j$-th Betti
    number} is $\beta_j(X):=\rk H^j(M(\scr A),\mat Z)$. The space $X$
  is called {\em minimal} if it is homotopy equivalent to a CW-complex
  with $\beta_j(X)$ cells of dimension $j$, for all $j\geq 0$. Such a
  CW-complex is also called minimal.
\end{definition}

\begin{teo}[Randell \cite{MR1900880}, Dimca and Papadima \cite{MR2018927}]\label{teo:mini_arr}
  The space $M(\scr A)$ is minimal. 
\end{teo}

\begin{cor}\label{cor:torsion}
  The cohomology groups $H^k(M(\scr A),\mat Z)$ are torsion-free.
\end{cor}
\begin{proof}
  Theorem \ref{teo:mini_arr} asserts the existence of a minimal
  complex for $M(\scr A)$. The (algebraic) boundary maps of the chain complex constructed from
  this minimal complex are all zero, thus torsion cannot arise in homology.
\end{proof}

  Corollary \ref{cor:torsion} 
  can
  be traced back to the seminal work of Brieskorn
  \cite{brieskorn1971gt}, where also the following other important fact about the
  cohomology of affine arrangements of hyperplanes was proved.

\begin{teo}[Brieskorn \cite{brieskorn1971gt}]\label{teo:brieskorn}
  Let $\scr A$ be a finite affine hyperplane arrangement. Then, for every $p \in \mat N$
  \begin{displaymath}
    H^p(M(\scr A); \mat Z) \cong \bigoplus_{X \in \call L(\scr A)_p} H^p(M(\scr A_X); \mat Z),
  \end{displaymath}
  where $\call L(\scr A)_p = \{X \in \call L(\scr A) \mid \codim(X) = p\}$.
\end{teo}

Intimely related with this torsion-freeness is the fact that it is
enough to compute de Rham cohomology in order to know the cohomology
with integer coefficients, the so-called {\em Orlik-Solomon algebra}
introduced in \cite{orlik1980cat}. Here, too, no broken circuit sets enter the
picture as most handy combinatorial invariants.

\begin{teo}\label{teo:poin_hyp}
  Let $\scr A$ be a complex central hyperplane arrangement,
  then the Poincar\'e polynomial of $M(\scr A)$ satisfies
  \begin{displaymath}
   P_{\scr A}(t):=\sum_{j\geq 0} \rk H^j(M(\scr A); \mat Z)\, t^j   
   = \sum_{j\geq 0} |\nbc_j(\scr A)|\,t^j.
  \end{displaymath}
\end{teo}
\begin{oss}
  In particular, the numbers $\vert \nbc_k(\scr A) \vert$ do not depend on the chosen
  ordering of $\scr A$. 
\end{oss}

\begin{oss}[\cite{jambuterao}]
  Combining Theorem \ref{teo:brieskorn} with Theorem \ref{teo:poin_hyp} we
  get the following formula for the Poincar\'e polynomial of
  the complement of an arbitrary finite affine complex arrangement:
  \begin{displaymath}
    P_{\scr A}(t):=\sum_{X \in \call L(\scr A)} |\nbc_{\codim X}(\scr A_X)|\, t^{\codim X}.
  \end{displaymath}
\end{oss}

We now turn to a special class of arrangements in complex space.

\begin{definition}
An arrangement $\scr A$ in $\mat C^d$ is called \emph{complexified} if
every hyperplane $H \in \scr A$ is the complexification of a real hyperplane,
i.e. if there is $\alpha_H\in (\mat R^d)^*$ and $a_H\in \mat R$ with
\begin{displaymath}
  H = \{ x\in \mat C^d \mid \alpha_H(\Re(x))+i\alpha_H(\Im(x))=a_H\}.
\end{displaymath}

\end{definition}

Let $\scr A$ be a complexified arrangement and consider its real part
\begin{displaymath}
  \scr A_{\mat R} = \{H \cap \mat R^d \mid  H \in \scr A\}, 
\end{displaymath}
an arrangement of hyperplanes in $\mathbb R^d$.
Notice that $\call L(\scr A) \cong \call L(\scr A_{\mat R})$
and therefore $\nbc(\scr A) = \nbc(\scr A_{\mat R})$. 

If $\scr A$ is a complexified arrangement, one can use the
combinatorial structure of $\scr A_{\mat R}$ to study the topology of $M(\scr A)$. Therefore we will write $\call F(\scr A) = \call F(\scr A_{\mat R})$,
$\call T(\scr A)=\call T(\scr A_{\mathbb R})$.

\subsubsection{The homotopy type of complexified arrangements}

Using combinatorial data about $\scr A_{\mat R}$, Salvetti defined in
\cite{salvetti1987tcr} a cell complex which embeds in the complement
$M(\scr A)$ as a deformation retract. We explain Salvetti's construction.

\begin{definition}\label{def:composizionecamere}
  Given a face $F \in \call F(\scr A)$ and a chamber $C \in \call T(\scr
  A)$, define 
  $C_F \in \call T(\scr A)$ as the unique chamber such that, for $H\in
  \scr A$,
  \begin{displaymath}
    \gamma_{C_F}(H) = \left\{
      \begin{array}{ll}
      \gamma_F(H) & \mbox{ if } \gamma_F(H) \neq 0,\\
      \gamma_C(H) & \mbox{ if } \gamma_F(H) = 0.
      \end{array}
    \right.
  \end{displaymath}
\end{definition}

The reader may think of $C_F$ as the one, among the chambers adjacent
to $F$, that ``faces'' $C$.

\begin{definition}
  Consider an affine complexified locally finite arrangement $\scr A$ and define the
  \emph{Salvetti poset} as follows:
  \begin{displaymath}
    \mbox{Sal}(\scr A) = \{ [F, C] \mid F \in \call F(\scr A), C \in \call T(\scr A)\,
    F \leq C\},
  \end{displaymath}
  with the order relation
  \begin{displaymath}
    [F_1, C_1] \leq [F_2, C_2] \iff
    F_2 \leq F_1 \mbox{ and } (C_2)_{F_1} = C_1.
  \end{displaymath}
\end{definition}

\begin{definition}\label{def:sal}
 Let $\scr A$ be an affine complexified locally finite hyperplane arrangement.
 Its \emph{Salvetti complex} is $\call S(\scr A) = \Delta(\mbox{Sal}(\scr A))$.
\end{definition}

\begin{teo}[Salvetti \cite{salvetti1987tcr}]
  The complex $\call S(\scr A)$ is homotopically equivalent to the complement
  $M(\scr A)$. More precisely $\call S(\scr A)$ embeds in $M(\scr A)$ as a
  deformation retract.
\end{teo}

\begin{oss}\label{sal_cell}
  In fact, the poset $\mbox{Sal}(\scr A)$ is the face poset of a
  regular cell complex
(of which $\call S(\scr A)$ is the barycentric subdivision)
whose maximal cells correspond to the pairs
  \begin{displaymath}
    \{[P, C]\mid P \in \min\call F(\scr A),\, C \in \call T(\scr A)\}.
  \end{displaymath}
  It is this complex that Salvetti describes in \cite{salvetti1987tcr}. When we
  need to distinguish between the two complexes we will speak of
  \emph{cellular} and \emph{simplicial Salvetti complex}.
\end{oss}

\subsubsection{Minimality}

In the case of complexified arrangements, explicit constructions of
a minimal CW-complex for $M(\scr A)$ were given in \cite{MR2350466}
and in \cite{delucchi}.
We review the material of \cite[\S 4]{delucchi} that will be useful for our later purposes.


\begin{lemma}[{\cite[Theorem 4.13]{delucchi}}]\label{XC}
Let $\scr A$ be a central arrangement of real hyperplanes, let $B\in \call T(\scr
A)$ and let $\preceq$ be any linear extension of the poset $\call T(\scr A)_B$. The subset of $\call L (\scr A)$ given
by all intersections $X$ such that 
\begin{displaymath}
  S(C,C')\cap \scr A_X \neq \emptyset\quad
  \mbox{ for all } C' \prec C
\end{displaymath}
is an order ideal of $\call L(\scr A)$. In particular, it has a well defined and unique
minimal element we will call $X_C$.
\end{lemma}

\begin{oss}
  Note that $X_C$ depends on the choice of $B$ and of the linear extension of $\call T(\scr A)_B$.
\end{oss}

\begin{cor}\label{cor:minimalchamber}
  For all $C\in \call T(\scr A)$ we have
  \begin{displaymath}
    C=\min_{\preceq}\{K\in \call T(\scr A) \mid K_{X_C} = C_{X_C}\},
  \end{displaymath}
  where, for $Y\in \call L(\scr A)$ and $K\in \call T (\scr A)$, we
  define 
 $K_Y:=\sigma_{\scr A_Y}(K)$.
\end{cor}

Now recall the (cellular) Salvetti complex of Definition \ref{def:sal} and
Remark \ref{sal_cell}. In particular, its maximal cells correspond to the pairs
$[P, C]$ where $P$ is a point and $C$ is a chamber. When $\scr A$ is a central
arrangement, the maximal cells correspond to the chambers in $\call T(\scr A)$.
In this case we can stratify the Salvetti complex assigning to each chamber
$C \in \call T(\scr A)$ the corresponding maximal cell of $\call S(\scr A)$,
together with its faces. Let us make this precise.
\begin{definition}
  Let $\scr A$ be a central complexified hyperplane arrangement and
  write $\min \call F(\scr A) = \{P\}$. Define
  a stratification of the cellular Salvetti complex
  $\call S(\scr A) = \bigcup_{C \in \call T(\scr A)} \call S_C$ through
  \begin{displaymath}
    \call S_C : = \bigcup \left\{ [F, K] \in \mbox{Sal}(\scr A) \mid
    [F, K] \leq [P, C] \right\}.
  \end{displaymath}
\end{definition}

\noindent Given an arbitrary linear extension $(\call T(\scr A), \preceq)$ of
$\call T(\scr A)_B$, for all $C\in \call T(\scr A)$ define
\begin{displaymath}
  \call N_C := \call S_C \backslash \sx\bigcup_{D \prec C} \call S_D\dx.
\end{displaymath}
In particular the poset $\mbox{Sal}(\scr A)$ can be partitioned as
\begin{displaymath}
  \mbox{Sal}(\scr A) = \bigsqcup_{C \in \call T(\scr A)} \call N_C(\scr A).
\end{displaymath}

\begin{teo}[{\cite[Lemma 4.18]{delucchi}}]\label{teo:min_delu}
There is an isomorphism of posets
\begin{displaymath}
  \call N_C \cong \call F(\scr A^{X_C})^{op}
\end{displaymath}
where $X_C$ is the intersection defined via Lemma \ref{XC} by the
same choice of base chamber and of linear extension of $\call T(\scr A)_B$
 used to define the subposets $\call N_C$. 
\end{teo}

\begin{oss}
  The alternative proof given in \cite{delucchi} of minimality of $M(\scr A)$ for $\scr A$ a
  complexified central arrangement follows from Theorem
  \ref{teo:min_delu} by an application of Discrete Morse Theory (see
  Section \ref{sec:DMT}). Indeed, from a shelling order of $\call
  F(\scr A^{X_C})$ one can construct a sequence of cellular collapses
  of the induced subcomplex of $\call S_C$ that leaves only one
  `surviving' cell. Via the Patchwork Lemma (Lemma \ref{PWL} below) these
  sequences of collapses can be concatenated to give a sequence of
  collapses on the cell complex $\call S(\scr A)$. The resulting
  complex after the collapses has one cell for every $\call N_C$,
  namely $\vert \nbc(\call A) \vert = P_{\scr A}(1)$ cells, and is
  thus minimal.
\end{oss}

\begin{figure}[thp]
  \centering
  \subfigure[$\call S_B$ and $\call N_{C_1}$]{
    \begin{tikzpicture}
      \tikzstyle{every path} = [line width = 2pt]
      \tikzstyle{every node} = [circle, inner sep = 2pt, outer sep = 2 pt]

      \begin{scope}
	\tikzstyle{every node} = [fill, circle, inner sep = 2pt, outer sep = 2 pt]
	
	\path node [inner sep = 1pt] (P) 	at (   0,   0) {};
	\path node (B)	at (  -3,   0) {}
	      node (C1)	at (-1.5, 2.5) {}
	      node (C2)	at ( 1.5, 2.5) {}
	      node (C3)	at (-1.5,-2.5) {}
	      node (C4)	at ( 1.5,-2.5) {}
	      node (C5)	at (   3,   0) {};
      \end{scope}

      \path node [anchor = east] 		(LB)	at (  -3,   0) { $B$ }
	    node [anchor = south east] 	(LC1)	at (-1.5, 2.5) { $C_1$ }
	    node [anchor = south west]	(LC2)	at ( 1.5, 2.5) { $C_2$ }
	    node [anchor = north east]	(LC3)	at (-1.5,-2.5) { $C_3$ }
	    node [anchor = north west]	(LC4)	at ( 1.5,-2.5) { $C_4$ }
	    node [anchor = west]		(LC5)	at (   3,   0) { $C_5$ };
      \path node [anchor = south]		(H1)	at (   0,   4) { $H_1$ }
	    node [anchor = south west]	(H2)	at (   4,   3) { $H_2$ }
	    node [anchor = north west]	(H3)	at (   4,  -3) { $H_3$ };

      \path[->] (B)		edge [bend left = 30] (C1)
		(C4) 	edge [bend left = 30] (C5);

      \begin{scope}
	\tikzstyle{every path} = [line width = 1pt]
	\tikzstyle{every node} = [circle, inner sep = 2pt, outer sep = 2 pt]
	      
	\path [line width = 10, white, draw]	
			(-4, 3) -- (P) -- (4, -3)
			(-4, -3) -- (P) -- (4, 3)
			( 0, 4) -- (P) -- (0, -4);
	
	\path [draw]	(-4, 3) -- (P) -- (4, -3)
			(-4, -3) -- (P) -- (4, 3)
			( 0, 4) -- (P) -- (0, -4);
      \end{scope}

      \begin{scope}
	\tikzstyle{every path} = [line width = 10pt, white]

	\path[->] 
		  (C1)	edge [bend left = 30] (C2)
		  (C2) 	edge [bend left = 30] (C5)
		  (B) 	edge [bend left = 30] (C3)
		  (C3) 	edge [bend left = 30] (C4)
		  (C1)	edge [bend left = 30] (B)
		  (C5)	edge [bend left = 30] (C4);
      \end{scope}

      \path[->] (C1)	edge [bend left = 30] (C2)
		(C2) 	edge [bend left = 30] (C5)
		(B) 	edge [bend left = 30] (C3)
		(C3) 	edge [bend left = 30] (C4);

      \path[green, ->] 
		(C1)	edge [bend left = 30] (B)
		(C5)	edge [bend left = 30] (C4);

      \path[pattern = dots, draw = none] (  -3,   0)  
	    to [bend left = 30]	(-1.5, 2.5)
	    to [bend left = 30]	( 1.5, 2.5)
	    to [bend left = 30]	(   3,   0)
	    to [bend right = 30]	( 1.5,-2.5)
	    to [bend right = 30]	(-1.5,-2.5)
	    to [bend right = 30]	(  -3,   0);

      \path[green, fill, draw = none, opacity=.4] (  -3,   0)  
	    to [bend right = 30]	(-1.5, 2.5)
	    to [bend left = 30]	( 1.5, 2.5)
	    to [bend left = 30]	(   3,   0)
	    to [bend left = 30]	( 1.5,-2.5)
	    to [bend right = 30]	(-1.5,-2.5)
	    to [bend right = 30]	(  -3,   0);
    
    \end{tikzpicture}
  }\\
  \subfigure[$\call F(\scr A^{X_{C_1}})$]{
    \begin{tikzpicture}[scale=.7]
      \tikzstyle{every path} = [draw, line width = 1pt]
      \tikzstyle{every node} = [fill, circle, inner sep = 1.5pt, outer sep = 3pt]
      
      \path node (P)	at ( 0, 0) {}
	    node (F1)	at (-2, 3) {}
	    node (F2)	at ( 2, 3) {};

      \begin{scope}
	\tikzstyle{every node} = [fill = none]

	\path node [anchor = north]	(LP)	at ( 0, 0) { $P$ }
	      node [anchor = south]	(LF1)	at (-2, 3) { $F_1$ }
	      node [anchor = south]	(LF2)	at ( 2, 3) { $F_2$ };
      \end{scope}

      \path	(P) -- (F1)
	    (P) -- (F2);
    \end{tikzpicture}
  }\hfill
  \subfigure[$\call N_{C_1}$]{
    \begin{tikzpicture}[scale = .7]
      \tikzstyle{every path} = [draw, line width = 1pt]
      \tikzstyle{every node} = [fill, circle, inner sep = 1.5pt, outer sep = 3pt]
      
      \path node (C1)	at ( 0, 0) {}
	    node (F1)	at (-2,-3) {}
	    node (F2)	at ( 2,-3) {};

      \begin{scope}
	\tikzstyle{every node} = [fill = none]

	\path node [anchor = south]	(LC1)	at ( 0, 0) { $[P, C_1]$ }
	      node [anchor = north]	(LF1)	at (-2,-3) { $[F_1, C_1]$ }
	      node [anchor = north]	(LF2)	at ( 2,-3) { $[F_2, C_5]$ };
      \end{scope}

      \path	(F1) -- (C1)
		(F2) -- (C1);
    \end{tikzpicture}
  }
\caption{Example of stratification}
\label{fig:startification_hyperplane}
\end{figure}
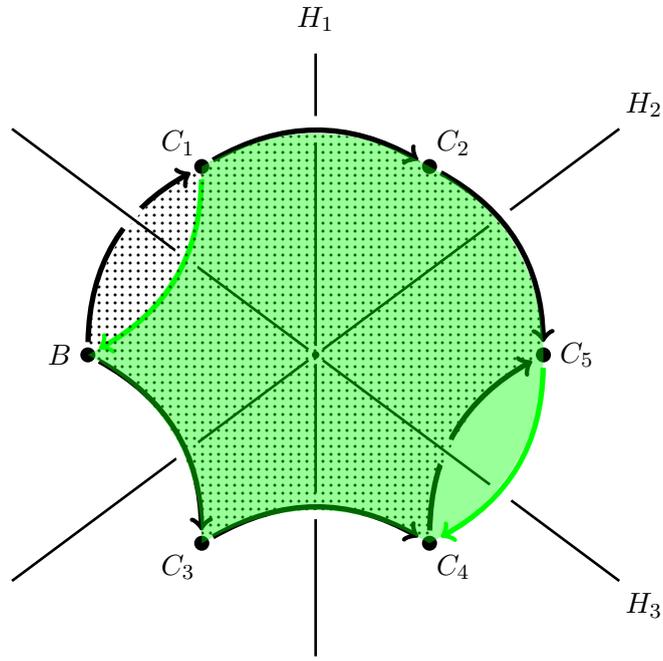
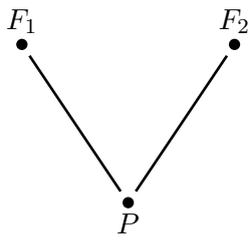
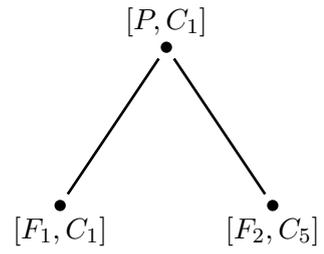

\begin{es}
  Consider the arrangement of Figure
  \ref{fig:startification_hyperplane}. We have
  \begin{displaymath}
    \call L(\scr A) = \{ \mat R^2, H_1, H_2, H_3, P \}
  \end{displaymath}
  where $P = H_1 \cap H_2 \cap H_3$. The chambers are ordered according to their indices,
  $B$ being the base chamber. Then, 
  $X_B = \mat R^2$, $X_{C_1} = H_3$, $X_{C_2} = H_1$, $X_{C_3} = H_2$,
  $X_{C_4} = X_{C_5} = P$.
  
  Recall the construction of the \emph{cellular} Salvetti Complex 
  (e.g. from \cite[Definition 2.4]{dantoniodelucchi}). 
  Figure \ref{fig:startification_hyperplane}.(a) shows in dotted black the stratum
  $\call S_B = \call N_B$ and in solid green the stratum $\call N_{C_1}$.
  The stratum $\call N_{C_1}$ consists of two $1$-dimensional faces and one
  $2$-dimensional face. Its poset structure is showed in 
  Figure \ref{fig:startification_hyperplane}.(c) and it is isomorphic,
  as a poset, to the order dual of $\call F(\scr A^{X_{C_1}})$, depicted in
  Figure \ref{fig:startification_hyperplane}.(b).
\end{es}

\section{Toric arrangements}
\label{sec:toric}

\subsection{Introduction}
\label{sec:toric_introduction}
We have presented arrangements of hyperplanes in affine space as
families of level sets of linear forms. Now, we want to explain in
which sense this idea generalizes to a toric setting. 


Our ambient spaces will be the \emph{complex torus} $(\mat C^*)^d$ and the
\emph{compact} (or \emph{real}) \emph{torus} $(S^1)^d$, where we
consider $S^1$ as the unit circle in $\mat C$. 
We consider {\em characters} of the torus, i.e., maps $\chi: (\mat C^*)^d \to \mat C^*$ given by
\begin{displaymath}
  \chi(x_1, \dots, x_d) = x_1^{\alpha_1}x_2^{\alpha_2} \cdots x_d^{\alpha_d}
  \mbox{ for an } \alpha = (\alpha_1, \dots, \alpha_d) \in \mat Z^d.
\end{displaymath}
Characters form a lattice, which we denote by $\Lambda$, under pointwise
multiplication. Notice that the assignment $\alpha \mapsto
x_1^{\alpha_1}\cdots x_d^{\alpha_d}$ provides an isomorphism
$\mat Z^d \to \Lambda$.

We consider subtori defined as level sets of characters, that is
hypersurfaces in $(\mat C^*)^d$ of the form
\begin{equation}\label{subtorus}
  K = \{x \in (\mat C^*)^d \mid \chi(x) = a\}
  \mbox{ with } \chi \in \Lambda, a \in \mat C^*.
\end{equation}

Notice that, if $a \in S^1$, the interesection $K \cap (S^1)^d$ is also a level set of a character (described
by the same equation).

\begin{definition}
\label{concrete}
  A \emph{(complex) toric arrangement} $\scr A$ in $(\mat C^*)^d$ is a finite set
  \begin{displaymath}
    \scr A = \{K_1, \dots, K_n\}
  \end{displaymath}
  of hypersufaces of the form \eqref{subtorus} in $(\mat C^*)^d$
\end{definition}


\begin{definition}
  Let $\scr A$ be a toric arrangement in $(\mat C^*)^d$. Its {\em complement}
  is 
  \begin{displaymath}
    M(\scr A) := (\mat C^*)^d\setminus \bigcup \scr A.
  \end{displaymath}
\end{definition}

\begin{definition}\label{def:concrete}
  A \emph{real toric arrangement} $\scr A$ in $(S^1)^d$ is a finite set
  \begin{displaymath}
    \scr A^c = \{K_1^c, \dots, K_n^c\}
  \end{displaymath}
  of hypersufaces $K_i^c$ in $(S^1)^d$ of the form \eqref{subtorus} with $a\in
  S^1$.
If a complex toric arrangement restricts to a real toric arrangement
on $(S^1)^d$ we will call $\scr A$ \emph{complexified}.
\end{definition}

We will often use this interplay between the complex 
and the `real' hypersurfaces in the same vein that one exploits
properties of the real part of complexified arrangements to gain
insight into the complexification.

\subsection{An abstract approach}

We now introduce an equivalent but more abstract approach to toric
arrangements. Being able to switch point of view according to the
situation will make our considerations below considerably more transparent.



\begin{definition}
  Let $\Lambda \cong \mat Z^d$ a finite rank lattice. The corresponding \emph{complex torus} is
  \begin{displaymath}
    T_\Lambda = \hom_{\mat Z}(\Lambda, \mat C^*).
  \end{displaymath}
  The \emph{compact (or real) torus} corresponding
  to $\Lambda$ is
  \begin{displaymath}
    T^c_\Lambda = \hom_{\mat Z}(\Lambda, S^1),
  \end{displaymath}
  where, again, $S^1:=\{z\in \mat C \mid |z| = 1\}$.
\end{definition}

The choice of a basis $\{u_1, \dots, u_d\}$ of $\Lambda$
gives isomorphisms 
\begin{equation}\label{isomorfismi}
  \begin{array}{c}
    \Phi:T_\Lambda \to (\mat C^*)^d\\
    \varphi \mapsto (\varphi(u_1), \dots, \varphi(u_d)) 
  \end{array}\quad\quad
  \begin{array}{c}
    \Phi^c:T^c_\Lambda \to (S^1)^d\\
    \varphi \mapsto (\varphi(u_1), \dots, \varphi(u_d))
  \end{array}\quad\quad
\end{equation}

\begin{oss}
  Consider a finite rank lattice $\Lambda$ and the corresponding torus $T_\Lambda$.
  The \emph{characters} of $T_\Lambda$ are the functions
  \begin{displaymath}
    \chi_\lambda:T_\Lambda \to \mat C^*,\quad
    \chi_\lambda(\varphi) = \varphi(\lambda) \mbox{ with } \lambda \in \Lambda.
  \end{displaymath}
  Characters form a lattice under pointwise multiplication, and 
  this lattice is naturally isomorphic to $\Lambda$.
  Therefore in the following we will identify the character lattice of $T_\Lambda$ with $\Lambda$.
\end{oss}

Now, the `abstract' definition of toric arrangements is the following.
\begin{definition}\label{def:toric_abstract}
  Consider a finite rank lattice $\Lambda$, a \emph{toric arrangement} in $T_\Lambda$
  is a finite set of pairs
  \begin{displaymath}
    \scr A = \{ (\chi_1, a_1),\ldots (\chi_n,a_n)\} \subset \Lambda \times \mat C^*.
  \end{displaymath}
  A toric arrangement $\scr A$ is called \emph{complexified} if $\scr
  A\subset \Lambda \times S^1$.
\end{definition}

\begin{oss}\label{oss:Kitor}
  The abstract definition is clearly equivalent to Definition
  \ref{def:concrete} via the isomorphisms in \eqref{isomorfismi} and by
  \begin{equation}\label{Ki}
    K_i:=\{x\in T_\Lambda\mid \chi_i(x)=a_i\}.
  \end{equation}
  Accordingly, we have 
    $M(\scr A):=T_\Lambda\setminus \bigcup\{K_1,\ldots, K_n\}$.
\end{oss}

\begin{definition}
  Let $\Lambda$ be a finite rank lattice. A {\em real} toric
  arrangement in $T_\Lambda^c$ is a finite set of pairs
  \begin{displaymath}
    \scr A^c = \{ (\chi_1, a_1),\ldots (\chi_n,a_n)\} \subset \Lambda
    \times  S^1.
  \end{displaymath}
\end{definition}

\begin{oss}
A complexified toric arrangement $\scr A$ in $T_\Lambda$ induces a
real toric arrangement $\scr A^c$ in $T_\Lambda^c$ with
\begin{displaymath}
  K^c_i := \{x \in T^c_\Lambda\mid \chi_i(x) = a_i\}.
\end{displaymath}
Furthermore, embedding $T^c_\Lambda \hookrightarrow T_\Lambda$ in the
obvious way, we have $K^c_i = K_i \cap T^c_\Lambda$ as in Definition \ref{def:concrete} .
\end{oss}

We now illustrate what has been proposed \cite{de2005geometry,mocicombinatorics} as the 
`toric analogue' of the intersection poset (see Definition \ref{def:IL}).
\begin{definition}\label{def:layerposet}
  Let $\scr A = \{(\chi_1, a_1), \dots, (\chi_n, a_n)\}$ be a toric
  arrangement on $T_\Lambda$.
  A {\em layer} of $\scr A$ is a connected component of a nonempty 
  intersection of some of the subtori $K_i$ (defined in Remark \ref{oss:Kitor}).
  The set of all layers of $\scr A$ ordered by reverse inclusion is
  the \emph{poset of layers} of the toric arrangement, 
  denoted by $\call C(\scr A)$.
\end{definition}

Notice that, as in the case of hyperplane arrangements, the torus $T_\Lambda$ itself is a layer,
while the empty set is not.

\begin{definition} Let $\Lambda$ be a rank $d$ lattice and let $\scr
  A$ be a toric arrangement on $T_\Lambda$. The {\em rank} of $\scr A$
  is $\rk(\scr A):=  \rk\gen{\chi \mid (\chi, a) \in \scr A} $.
\begin{itemize}
\item[(a)] A character $\chi \in \Lambda$ is called \emph{primitive} if,
  for all $\psi \in \Lambda$, $\chi=\psi^k$ only if $k\in \{-1,1\}$.
\item[(b)] The toric arrangement $\scr A$ is called {\em primitive} if for each
  $(\chi, a) \in \scr A$, $\chi$ is primitive.
\item[(c)] The toric arrangement $\scr A$ is called {\em  essential} if
  $\rk(\scr A)=d$.
\end{itemize}
\end{definition}

\begin{oss}\label{rem:assu}
For every non primitive arrangement there is a primitive arrangement
which has the same complement. Furthermore, if $\scr A$ is
a non essential arrangement, then there is an essential arrangement $\scr A'$
such that
\begin{displaymath}
  M(\scr A) \cong (\mat C^*)^{d-l} \times M(\scr A') \mbox{ where } l = \rk(\scr A').
\end{displaymath}
Therefore the topology of $M(\scr A)$ can be derived from the topology of
$M(\scr A')$.  
\end{oss}

In view of Remark \ref{rem:assu}, our study of the topology of
complements of toric arrangements will not loose in generality by
stipulating the next assumption.

\begin{assu}
  From now on we assume every toric arrangement to be \emph{primitive} and \emph{essential}.
\end{assu}

\subsubsection{Deletion and restriction}

Let $\Lambda $ be a finite rank lattice and $\scr A$ be a toric arrangement in $T_\Lambda$.

\begin{definition}\label{def:sublattice}
  For every sublattice $\Gamma \subseteq \Lambda$ we define the arrangement
  \begin{displaymath}
    \scr A_\Gamma = \{(\chi, a) \mid \chi \in \Gamma\},
  \end{displaymath}

\noindent for every layer $X\in \call C(\scr A)$ a sublattice
\begin{displaymath}
  \Gamma_X:=\{\chi\in \Lambda \mid \chi \mbox{ is constant on }
  X\}\subseteq \Lambda.
\end{displaymath}
\end{definition}

\begin{definition}
  Let $X$ be a layer of $\scr A$. We
  define toric arrangements
  \begin{displaymath}
    \scr A_X:=\scr A_{\Gamma_X}\mbox{ on } T_{\Gamma_X},\quad 
  \end{displaymath}
  and

\begin{displaymath}
\scr A^X:=\{K_i\cap X \mid X\not\subseteq K_i\}\mbox{ on the torus }X.   
\end{displaymath}
\end{definition}

\begin{oss}
  Notice that for a layer $X \in \call C(\scr A)$ and an hypersurface
  $K$ of $\scr A$, the interesection $K \cap X$ needs not to be connected.
  
  In general $K \cap X$ consist of several connected components, each of which
  is a level set of a character in the torus $X$. In particular
  $\scr A^X$ is a toric arrangement in the sense of Definition \ref{def:toric_abstract}
\end{oss}

\subsubsection{Covering space}
\label{sec:covering}

We now recall a construction of \cite{dantoniodelucchi} which we need in the following.
For more details we refer to \cite[\S 3.2]{dantoniodelucchi}.
Consider the covering map:
\begin{equation}\label{def:covering}
  \begin{array}{c}
    p: \mat C^d \cong \Hom_{\mat Z}(\Lambda; \mat C) \to \Hom_{\mat Z}(\Lambda; \mat C^*) = T_\Lambda\\
    \varphi \mapsto \mbox{exp}\circ\varphi
  \end{array}
\end{equation}
Notice that identifying $\Hom_{\mat Z}(\Lambda, \mat C) \cong \mat C^d$, $p$ becomes the
universal covering map
\begin{displaymath}
  (t_1, \dots, t_d) \mapsto (e^{2\pi i t_1}, \cdots, e^{2\pi i t_d})
\end{displaymath}
of the torus $T_\Lambda$. Also, this map restricts to a universal covering map
\begin{displaymath}
\mat R^d \cong \Hom_{\mat Z}(\Lambda; \mat R) \to \Hom_{\mat Z}(\Lambda, S^1) \cong (S^1)^d.
\end{displaymath}

Consider now a toric arrangement $\scr A$ on $T_\Lambda$. Its preimage through $p$ is
a locally finite affine hyperplane arrangement on $\Hom_{\mat Z}(\Lambda; \mat C)$
\begin{displaymath}
  \sopra{\scr A} = \{(\chi, a') \in \Lambda \times \mat C\mid
  (\chi, e^{2\pi i a'}) \in \scr A\}.
\end{displaymath}
If we write it in coordinates, $\sopra{\scr A}$ becomes the arrangement on $\mat C^d$
defined as
\begin{displaymath}
  \sopra{\scr A} = \{ H_{\chi, a'}\mid (\chi, e^{2\pi i a'}) \in \scr A\}
  \mbox{ with } H_{\chi, a'} = \{  x \in \mat C^n \mid
  \sum \alpha_ix_i = a' \},
\end{displaymath}
where we expanded $\chi(x)=x_1^{\alpha_1}\cdots
x_d^{\alpha_d}$.

\begin{oss}
  If the toric arrangement $\scr A$ is complexified, so is the
  hyperplane arrangement $\sopra{\scr A}$.
\end{oss}

\subsection{Combinatorics}

As in the case of hyperplanes, one would like to describe
the topology of the complement in terms of the combinatorics of the
arrangement.  

\begin{lemma}
  Let $\scr A$ be a toric arrangement, $X\in \call C(\scr A)$ a
  layer. Then the subposet $\call C(\scr A)_{\leq X}$ is the
  intersection poset of a central hyperplane arrangement $\scr
  A[X]$. If $\scr A$ is complexified, then $\scr A[X]$ is, too.
\end{lemma}
\begin{proof}
  This is implicit in much of \cite{de2005geometry,mocicombinatorics}, the proof follows by
  lifting the layer $X$ to $\sopra{\scr A}$. A formally precise
  definition of $\scr A[Y]$ can also be found in Section
  \ref{sec:strata} below.
\end{proof}

In other words, 
lower intervals of posets of layers are intersection
lattices of (central) hyperplane arrangements. The following definition is then natural.

\begin{definition}[\cite{de2005geometry,mocicombinatorics}]\label{def:lnbc}
  Let $\scr A$ be a toric arrangement of rank $d$ and let us fix a total ordering on $\scr A$. A \emph{local no broken circuit set} of $\scr A$ is a pair
  \begin{displaymath}
    (X, N) \mbox{ with } X \in \call C(\scr A), N \in \nbc_k(\scr
    A(X))
    \mbox{ where }k=d-\dim X
  \end{displaymath}
  We will write $\scr N$ for the set of local non broken
  circuits, and partition it into subsets
  \begin{displaymath}
    \scr N_j = \{(X,N) \in \scr N\mid \dim X = d - j\}.
  \end{displaymath}
\end{definition}

\begin{oss} 
  Let $X\in \call C(\scr A)$ and $N\subseteq \scr A(X)$. If we consider the `list' $\scr X$ of all pairs
  $(\chi_i,a_i)$ with ${\chi_i}_{\vert X}\equiv a_i$, then the elements
  of $N$ index a `sublist' $\scr X_N$. Then, $(X,N)$ is a local no
  broken circuit set if and only if $\scr X_N$ is a basis of $\scr X$
  with no {\em local external activity} in the sense of d'Adderio and
  Moci \cite[Section 5.3]{mocidadd}
\end{oss}


\subsection{Cohomology}

The cohomology (with complex coefficients) of the complements of toric arrangements
was studied by Looijenga \cite{looij} and De Concini and Procesi \cite{de2005geometry}. 




\begin{teo}[{\cite[Theorem 4.2]{de2005geometry}}]
  Consider a toric arrangement $\scr A$. The Poincar\'e polynomial of
  $M(\scr A)$ can be
  expressed as follows:
  \begin{displaymath}
    P_{\scr A}(t) = \sum_{j = 0}^\infty \dim H^j(M(\scr A); \mat C)\, t^j =
    \sum_{j=0}^\infty |\scr N_j|\, (t+1)^{k-j}\, t^j.
  \end{displaymath}
\end{teo}

This result was reached in \cite{de2005geometry} by computing de Rham
cohomology, in \cite{looij} via spectral sequence computations. In the special case of (totally) unimodular arrangements,
De Concini and Procesi also determine the algebra structure of
$H^*(M(\scr A),\mat C)$ by formality of $M(\scr A)$
\cite[Section 5]{de2005geometry}.

\subsection{The homotopy type of complexified toric arrangements}

From now on in this
paper we will think of $\scr A$ as being a complexified (primitive,
essential) toric arrangement. 


The complement of a 
complexified toric arrangement $\scr A$ has the homotopy type of a finite
cell complex, defined from the stratification of the real torus
$T_\Lambda$ into {\em chambers} and
{\em faces} induced by the associated `real' arrangement $\scr A^c$.



\begin{definition}
  Consider a complexified toric arrangement $\scr A = \{(\chi_1, a_1), \dots, (\chi_n, a_n)\}$,
  its \emph{chambers} are the connected components of 
  $M(\scr A^c)$. 
  We denote the set of chambers of $\scr A$ by $\call T(\scr A)$.
  
  The \emph{faces} of $\scr A$ are the connected components of the intersections
  \begin{displaymath}
    \operatorname{relint}(\bar C \cap X) \mbox{ with } C \in \call T(\scr A)\, X \in \call C(\scr A).
  \end{displaymath}
  
  The faces of $\scr A$ are the cells of a polyhedral complex, which we denote by $\call D(\scr A)$.
\end{definition}

The topology of a (non regular) polyedral complex is encoded in an acyclic category, called
the \emph{face category} of the complex (see \cite[\S 2.2.2]{dantoniodelucchi} for some details
on face categories, our Section \ref{sec:DMT} below for some basics about 
acyclic categories,  \cite{kozlov2007combinatorial} for a more
comprehensive treatment).

\begin{definition}
  The face category of a complexified toric arrangement is $\call F(\scr A) = \call F(\call D(\scr A))$,
  i.e. the face category of the polyhedral complex $\call D(\scr A)$.
\end{definition}

The lattice $\Lambda$ acts on $\mat C^n$ and on $\mat R^n$ as the
group of automorphisms of the covering map $p$ of
 \eqref{def:covering} above. Consider now the map $q:\call F(\sopra{\scr A}) \to \call F(\scr A)$
induced by $p$.
\begin{prop}[{\cite[Lemma 4.8]{dantoniodelucchi}}]\label{prop:quoziente}
  Let $\scr A$ be a complexified toric arrangement.
  The map $q:\call F(\sopra{\scr A}) \to \call F(\scr A)$ induces an isomorphism of
  acylic categories
  \begin{displaymath}
    \call F(\scr A) \cong \call F(\sopra{\scr A})/\Lambda.
  \end{displaymath}
\end{prop}

\subsubsection{The Salvetti category}

Recall that the Salvetti complex for affine hyperplane arrangements makes use
of the operation of Definition \ref{def:composizionecamere}. We need a suitable analogon
for toric arrangements.

\begin{prop}[{\cite[Proposition
    3.12]{dantoniodelucchi}}]
  Let $\Lambda$ be a finite rank lattice, $\Gamma$ a sublattice of
  $\Lambda$. Let $\scr A$ a complexfied toric arrangement on
  $T_\Lambda$ and recall the arrangement $\scr A_\Gamma$ from
  Definition \ref{def:sublattice}.
  The projection $\pi_\Gamma: T_\Lambda \to T_\Gamma$ induces a morphism of
  acyclic categories
  \begin{displaymath}
    \pi_\Gamma: \call F(\scr A) \to \call F(\scr A_\Gamma).
  \end{displaymath}
\end{prop}

Consider now a face $F \in \call F(\scr A)$. We associate to it the sublattice
\begin{displaymath}
  \Gamma_F = \{ \chi \in \Lambda\mid \chi \mbox{ is constant on } F\} \subseteq \Lambda
\end{displaymath}

\begin{definition}
  Consider a toric arrangement $\scr A$ on $T_\Lambda$ and a face $F \in \call F(\scr A)$.
  The \emph{restriction} of $\scr A$ to $F$ is the arrangement $\scr A_F = \scr A_{\Gamma_F}$ 
  on $T_{\Gamma_F}$.
\end{definition}

We will write $\pi_F = \pi_{\Gamma_F}: \call
F(\scr A) \to \call F(\scr A_F)$.

\begin{definition}[{\cite[Definition 4.1]{dantoniodelucchi}}]
  Let $\scr A$ be a toric a arrangement on a complex torus $T_\Lambda$. The \emph{Salvetti category}
  of $\scr A$ is the category $\Sal \scr A$ defined as follows.
  \begin{itemize}
  \item[(a)] The objects are the morphisms in $\call F(\scr A)$ between faces and chambers:
    \begin{displaymath}
      \Obj (\Sal \scr A) = \{ m:F \to C \mid m \in \Mor(\call F(\scr A)),\, C \in \call T(\scr A) \}.
    \end{displaymath}
  \item[(b)] The morphisms are the triples $(n, m_1, m_2): m_1 \to m_2$, where
    $m_1:F_1 \to C_1, m_2:F_2 \to C_2 \in \Obj(\Sal \scr A)$, $n: F_2 \to F_1 \in \Mor(\call F(\scr A))$
    and $m_1, m_2$ satisfy the condition:
    \begin{displaymath}
      \pi_{F_1}(m_1) = \pi_{F_1}(m_2).
    \end{displaymath}
  \item[(c)] Composition of morphisms is defined as:
    \begin{displaymath}
      (n',m_2,m_3) \circ (n, m_1, m_2) = (n \circ n', m_1, m_3),
    \end{displaymath}
    whenever $n$ and $n'$ are composable.
  \end{itemize}
\end{definition}

\begin{oss}
  The Salvetti category is an acyclic category in the sense of Definition \ref{def:acyclic}.
\end{oss}

\begin{definition}
  Let $\scr A$ be a complexified toric arrangement; its \emph{Salvetti complex} is the nerve
  $\call S(\scr A) = \Delta(\Sal \scr A)$.
\end{definition}

\begin{teo}[{\cite[Theorem 4.3]{dantoniodelucchi}}]
  The Salvetti complex $\call S(\scr A)$ embeds in the complement $M(\scr A)$ as a deformation retract.
\end{teo}

\begin{oss}\label{oss:unsub}
  As for the case of affine arrangements, the Salvetti category is the face category
  of a polyhedral complex, of which the toric Salvetti complex is a subdivision.
  If we need to distinguish between the two, we will call the first
  \emph{cellular Salvetti complex} and the second \emph{simplicial Salvetti complex}.
\end{oss}

\section{Discrete Morse theory}
\label{sec:DMT}

Our proof of minimality will consist in describing a sequence of
cellular collapses on the toric Salvetti complex, which is not
necessarily a regular cell complex. We need thus to estend discrete
Morse theory from posets to acyclic categories.

The setup used in the textbook of Kozlov
\cite{kozlov2007combinatorial} happens to lend itself very nicely to
such a generalization - in fact, once the right definitions are made,
even the proofs given in \cite{kozlov2007combinatorial} just need some
minor additional observation. 

\begin{definition}\label{def:acyclic}
  An {\em acyclic category} is a small category where the only
  endomorphisms are the identities, and these are the only invertible
  morphisms.

  An {\em indecomposable morphism} in an acyclic category is a morphism
  that cannot be written as the composition of two nontrivial
  morphisms. The {\em length} of a morphism $m$ in an acyclic category is
  the maximum number of members in a decomposition of $m$ in
  nontrivial morphisms. The {\em height} of an acyclic category is the
  maximum of the lengths of its morphisms: here we will restrict
  ourselves to acyclic categories of finite height.

A {\em rank function} on an acyclic category $\call C$ is a function
$\rk:\ob(\call C)\to \mathbb N$ such that  $\rk^{-1}(0)\neq \emptyset$
and such that for every indecomposable morphism $x\to y$,
$\rk(x)=\rk(y)-1$. An acyclic category is called {\em ranked} if it
admits a rank function.

  A {\em linear extension} $\prec$ of an acyclic category is a total
  order on its set of objects,
  such that
  \begin{displaymath}
    \Mor(x,y) \neq \emptyset \Longrightarrow x \prec y.
  \end{displaymath}

\end{definition}

\begin{oss}[Acyclic categories and posets]\label{acicats:posets}
  Every partially ordered set can be viewed as an acyclic category
  whose objects are the elements of the poset and where 
  $\vert\Mor(x,y)\vert =1$ if $x\leq y$, $\vert \Mor(x,y)\vert = 0$
  else (see \cite[Exercise 4.9]{kozlov2007combinatorial}).  

  Conversely, to every acyclic category $\call C$ is naturally
  associated a partial order on the set $\ob(\call C)$ defined by
  $x\leq y$ if and only if $\Mor(x,y) \neq \emptyset $. We denote by
  $\underline{\call C}$ this poset and by $\underline{\,\cdot\,} : \call
  C\to \underline{\call C}$ the natural functor, with
  $\underline{\call C}$ viewed as a category as above. We say $\call
  C$ {\em is} a poset if this functor is an isomorphism.  

  In the following sections we will freely switch between the
  categorical and set-theoretical point of view about posets.
 
\end{oss}


\begin{oss}[Face categories] The acyclic categories we will be
  concerned with will arise mostly as face categories of polyhedral
  complexes. Intuitively, we call
  polyhedral complex a CW complex $X$ whose cells are polyhedra, and such
  that the attaching maps of a cell $x$ restrict to homeomorphisms on every boundary
  face of $x$. The face category  then has an object for every cell of
  $X$ and an arrow $x\to y$ for every boundary cell of $y$ that is
  attached to $x$. See \cite[Definition 2.6 and 2.8]{dantoniodelucchi} for the precise definition.

  Notice that the face category of a polyhedral complex is naturally
  ranked by the dimension of the cells.
\end{oss}

\begin{oss}[Terminology]\label{oss:acycliccat}
  We take the term {\em acyclic category} from
  \cite{kozlov2007combinatorial}.
  The same name, in other contexts, is given to categories with
  acyclic nerve. The reader be warned: acyclic categories as defined here must by no means
  have acyclic nerve.

  On the other hand, the reader should be aware that what we call ``acyclic category'' appears in the
  literature also as {\em loopless category} or as {\em scwol} (for
  ``small category without loops'').
\end{oss}

The data about the cellular collapses that we will perform are stored
in so-called {\em acyclic matchings}.

\begin{definition}
A {\em matching} of an acyclic category $\call C$ is a set $\match$ of
indecomposable morphisms such that, for every $m,\, m'\in \match$, the
sources and the targets of $m$ and $m'$ are four distinct objects of
$\call C$. 
A {\em cycle} of a matching $\match$ is an ordered sequence of morphisms
$$a_1b_1a_2b_2\cdots a_nb_n$$
where
\begin{itemize}
\item[(1)] For all $i$, $a_i\not\in \match$ and $b_i\in \match$, 
\item[(2)] For all $i$, the targets of $a_i$ and $b_i$ coincide and the
sources of $a_{i+1}$ and $b_{i}$ coincide - as do the sources of
$a_1$ and $b_n$.
\end{itemize}

A matching $\match$ is called {\em acyclic} if it has no cycles. A {\em
  critical element} 
of $\match$ is any object of $\call C$
that is neither
source nor target of any $m\in \match$.
\end{definition}

\begin{lemma}\label{lem:linext_acmat}
  A matching $\match$ of an acyclic category $\call C$
  is acyclic if and only if
  \begin{itemize}
  \item[(a)]     for all $x,y\in \ob\call C$, $m \in \match\cap \Mor(x,y)$ implies $\Mor(x,y)=\{m\}$;
  \item[(b)] 
  there is a linear extension of
    $\call C$ where source and target of every $m\in \match$ are
    consecutive.
  \end{itemize}

\end{lemma}
\begin{proof} Recall from Remark \ref{acicats:posets} the poset
  $\underline{\call C}$, and notice that for every matching $\match$
  of $\call C$, the set $\underline{\match}$ is a matching of
  $\underline{\call C}$. Moreover, by Theorem 11.1 of
  \cite{kozlov2007combinatorial}, condition (b) above is
  equivalent to  $\underline{\match}$ being acyclic.

To prove the statement, let first $\match$ be a matching of $\call C$ satisfying (a) and
(b). Because of (a), every cycle of $\match$ maps to a cycle of 
$\underline{\match}$. Since
$\underline{\match}$ is acyclic because of (b), $\match$ must be acyclic too.

Let now $\match$ be an acyclic matching of $\call C$,
then  $\underline{\match}$ must be acyclic, thus (b) holds. If (a)
fails, say because of some $x,y\in \ob\call C$ with
$\Mor(x,y)\supseteq\{m,m'\}$ and $m\in \match$, then $m'\not\in\match$
(because $\match$ is a matching) and the sequence $m'm$ is a cycle of $\match$,
contradicting the assumption.

\end{proof}

A handy tool for dealing with acyclic matchings is
the following result, which generalizes 
\cite[Theorem 11.10]{kozlov2007combinatorial}.

\begin{lem}[Patchwork Lemma]\label{PWL} Consider a functor of acyclic categories
$$\varphi:\call C \to \call C'$$
and suppose that for each object $c$ of $C'$ an acyclic matching $\match_c$
of $\varphi^{-1}(c)$ is given.

Then the matching $\match:=\bigcup_{c\in \operatorname{Ob}\call C'}\match_c$ of
$\call C$ is acyclic.
\end{lem}
\begin{proof} We apply Lemma \ref{lem:linext_acmat}. 
  Since $\Mor_{\varphi^{-1}(c)}(x,y)=\Mor_{\call C}(x,y)$ for all $c\in \ob\call C'$ and all $x,y\in \ob
  (\varphi^{-1}(c))$, condition (a) holds for $\match$ because it holds
  for $\match_c$. 

  Property (b) for $\match$ is proved via the linear extension of $\call C$
  obtained by concatenation of the linear extensions given by the
  $\match_c$ on the categories $\varphi(c)$.
\end{proof}

The topological gist of Discrete Morse Theory is the so-called
``Fundamental Theorem'' (see e.g. \cite[\S 11.2.2]{kozlov2007combinatorial}). Here we state the
part of it that will be needed below.

\begin{teo} \label{dmt_main} Let $\call F$ be the face category of a
  finite polyhedral complex
  $X$, and let $\match$ be an acyclic matching of $\call F$. Then $X$ is homotopy
  equivalent to a CW-complex $X'$ with, for all $k$, one cell of dimension $k$ for every critical element  of $\match$ of rank $k$.
\end{teo}

\begin{proof} A proof can be obtained applying \cite[Theorem
  11.15]{kozlov2007combinatorial} to the filtration of $X$ with $i$-th
  term $F_i(X)=\bigcup _{j\leq i} x_j$, where $x_0,x_1,\ldots$ is an
  enumeration of the cells of $X$ corresponding to a linear extension
  of $\call F(X)$ in which source and target of every $m\in \match$ are
  consecutive (such a linear extension exists by Lemma \ref{lem:linext_acmat}.(b)).
\end{proof}

\begin{oss}\label{oss:perfect} Let $\match$ be an acyclic matching of
  a polyhedral complex $X$.
  \begin{itemize} 
  \item[(i)] 
  The
  boundary maps of the complex $X'$ in Theorem \ref{dmt_main} can be
  explicitely computed by tracking the individual collapses, as in \cite[Theorem
  11.13.(c)]{kozlov2007combinatorial}.
  \item[(ii)] We will call $\match$ {\em perfect}
    if the number of its critical elements of rank $k$ is
    $\beta_k(X)$, the $k$-th Betti number of $X$. Note that if the
    face category of a complex $X$ admits a perfect acyclic matching,
    then $X$ is minimal in the sense of \cite{MR2018927}. 
  \end{itemize} 
\end{oss}

\section{Stratification of the toric Salvetti complex}
\label{section:strata}
We now work our way towards the proof of minimality of complements of toric arrangements.
We start here by defining a stratification of the toric Salvetti Complex, in which each stratum
corresponds to a local non broken circuit. Then, in the next section,
we will exploit the structure of this stratification to define a perfect
acyclic matching on the Salvetti Category.

\subsection{Local geometry of complexified toric arrangements}
\label{sec:strata}

We start by introducing the key combinatorial tool in order to have a `global' control of the
local contributions.

  Consider 
  a rank $d$ complexified toric arrangement $\scr A =
  \{(\chi_1, a_1), \dots, (\chi_n, a_n)\}$. 
  As usual write $\chi_i(x) = x^{\alpha_i}$ for $\alpha_i \in \mat Z^d$ and $K_i = \{x \in T_\Lambda\mid \chi_i(x) = a_i \}$.


  Define
  \begin{displaymath}
    \scr A_0 := \left\{H_i =\ker\,\scal{\alpha_i}{\cdot}\mid i = 1, \dots, n \right\},
  \end{displaymath}
 a central hyperplane arrangement in $\mat R^d$.

From now on, fix a chamber $B\in \call T(\scr A_0)$
and a linear extension $\prec_0$ of $\call T(\scr A_0)_B$.

Next, we introduce some central arrangements associated with the
`local' data.

\begin{definition}
For every face $F \in \call F(\scr A)$ define the arrangement
\begin{displaymath}
  \scr A[F] = \{ H_i \in \scr A_0\mid \chi_i(F) = a_i\}.
\end{displaymath}

 If $Y\in \call C(\scr A)$ define
\begin{displaymath}
  \scr A[Y] = \{H_i \in \scr A_0 \mid Y\subseteq K_i\}.
\end{displaymath}
\end{definition}

\begin{oss}
  The linear extension $\prec_0$ of $\call T(\scr A_0)_B$ induces as in
  Proposition \ref{prop_indext} linear extensions $\prec_F$ of $\call
  T(\scr A[F])_{B_F}$ and $\prec_{Y}$ of
  $\call T(\scr A[Y])_{B_Y}$, for every $F\in \call F(\scr A)$ and
  every $Y\in \call C(\scr A)$. 

Moreover, for $F\in\call F(\scr A)$ and $C,C'\in \call T(\scr A[F])$ we denote
  by $S_F(C,C')$ the set of separating hyperplanes of the arrangement
  $\scr A[F]$, as introduced in
  Definition \ref{def_sep}.
\end{oss}

\begin{definition}\label{def:xtilda}
Given $Y\in \call C(\scr A)$ 
let $\widetilde {Y} \in \call L(\scr A_0)$ be
defined as 
\begin{displaymath}
  \widetilde{Y} := \bigcap_{Y\subseteq K_i} H_i.
\end{displaymath}
  Moreover, ror $C\in \call T(\scr A[Y])$ let $X(Y,C)\supseteq Y$
  be the layer determined by the intersection defined by Lemma \ref{XC} from $\prec_Y$. Analogously,
  for $C\in \call T(\scr A[F])$ let $X(F,C)$ be defined with respect
  to $\prec_F$.

We write $\widetilde{X}(Y,C)$ 
and $\widetilde{X}(F,C)$ 
for the corresponding elements of $\call L(\scr A [Y])$ and $\call L (\scr A[F])$. 
\end{definition}

\begin{definition}
  \label{def:ipsilon}
  Let 
  \begin{displaymath}
    \scr Y :=\{(Y,C) \mid Y\in \call C(\scr A),\, C\in \call T(\scr
    A[Y]),\, X(Y,C)=Y\}.
  \end{displaymath}
  For $i=0,\ldots,d$ let $\scr Y_i:=\{(Y,C)\in \scr Y \mid \dim(Y)=i\}$.
\end{definition}

\begin{figure}[tp]
  \centering

  \subfigure[A toric arrangement]{
  \label{fig:toricarrsub}
  \begin{tikzpicture}[scale = .95]
    \tikzstyle{every path} = [draw, line width = 1pt]
    \tikzstyle{every node} = [fill, circle, inner sep = 1.5pt, outer sep = 3pt]
    
    \begin{scope}
      \tikzstyle{every path} = [draw, line width = 2pt, style = loosely dashed, opacity = .5]
      
      \path (0,0) -- (4,0) -- (4,4) -- (0,4) -- cycle;
    \end{scope}

    \begin{scope}
      \tikzstyle{every path} = [draw, line width = 1pt]
      
      \path	(0,0) -- (4,4);
      \path	(0,4) -- (4,0);
      \path	(0,0) -- (0,4)
		(4,0) -- (4,4);
    \end{scope}
    
    \begin{scope}
      \tikzstyle{every node} = [fill = none]
      
      \path node [anchor = east] (P) at (2,2) {$P$}
	    node [anchor = south west] (Q) at (4,4) {$Q$}
	    node [anchor = south east] (F) at (3,3) {$F$}
	    node [anchor = east] (K1) at (0,2) {$K_1$}
	    node [anchor = north west] (K2) at (.5,.7) {$K_2$}
	    node [anchor = south west] (K3) at (.5,3.2) {$K_3$};
    \end{scope}

    \begin{scope}
      \tikzstyle{every node} = [draw = none, fill, circle, inner sep = 2pt]
      \tikzstyle{every path} = [draw, line width = 2pt]

      \path node (PP) at (2,2) {}
	    node (Q1) at (0,0) {}
	    node (Q2) at (4,0) {}
	    node (Q3) at (0,4) {}
	    node (Q4) at (4,4) {};

      \path (PP) -- (Q4);
    \end{scope}

  \end{tikzpicture}
  }\hfill
  \subfigure[The arrangement {\protect $\scr A[P]$}]{
    \label{fig:arragementap}
      \begin{tikzpicture}[scale = .95]
	\tikzstyle{every path} = [draw, line width = 1.5pt]
    
	\path[pattern = vertical lines, draw = none] (2,2) -- (0,4) -- (4,4) -- cycle;
      
	\path node [fill = white] at (2,3.5) {$D_0$};

	\path 		(0,0) -- (4,4);
	\path 		(0,4) -- (4,0);
	
      \end{tikzpicture}
  }\hfill
  \subfigure[The arrangement {\protect $\scr A[Q]$}]{
    \label{fig:arragementaq}
    \begin{tikzpicture}[scale = .95]
      \tikzstyle{every path} = [draw, line width = 1.5pt]

      \path[pattern = dots, draw = none] (2,2) -- (0,4) -- (2,4) -- (2,2);
      \path[pattern = horizontal lines, draw = none] (2,2) -- (2,4) -- (4,4) -- cycle;        

      \path node [fill = white] at (1.3,3.7) {$D_1$}
	    node [fill = white] at (2.7,3.7) {$D_2$};

      \path		(0,0) -- (4,4);
      \path 		(0,4) -- (4,0);
      \path 		(2,0) -- (2,4);
    \end{tikzpicture}

  }
  \caption{A toric arrangement and its associated hyperplane arrangements}
  \label{fig:toricarr}
\end{figure}
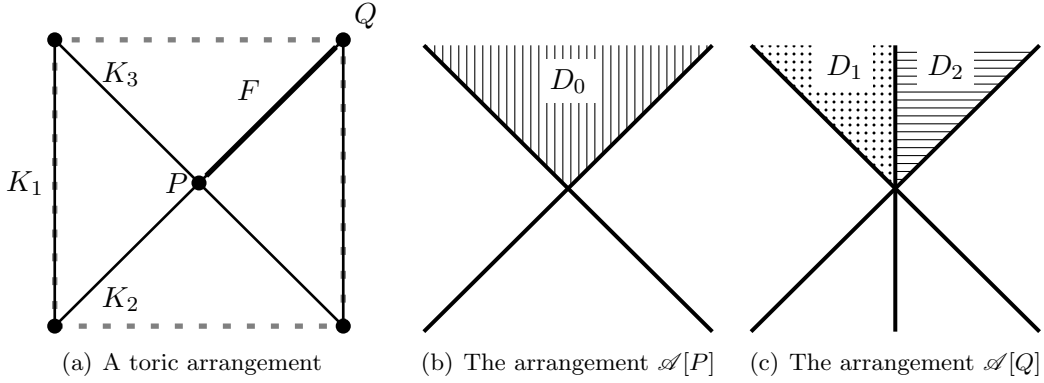

\begin{es}
  Consider the toric arrangement $\scr A = \{ (x,1), (xy^{-1}, 1), (xy, 1)\}$ of Figure
  \ref{fig:toricarrsub}. In this and in the following pictures we consider the compact
  torus $(S^1)^2$ as a quotient of the square. 
  Therefore we draw toric arrangements in a square (pictured with a dashed line),
  where the opposite sides are identified.

  The layer poset consists of the following elements
  \begin{displaymath}
    \call C(\scr A) = \{ P, Q, K_1, K_2, K_3, (\mat C^*)^2 \}.
  \end{displaymath}
  Figures \ref{fig:arragementap} and \ref{fig:arragementaq} show respectively the
  arrangements $\scr A[P]$ and $\scr A[Q] = \scr A_0$.

  Let $\scr Y$ as in Definition \ref{def:ipsilon}. There is one element $(P, D_0) \in \scr Y$
  and two elements $(Q, D_1), (Q, D_2) \in \scr Y$.
  Furthermore we have an element for each $1$-dimensional layer $(K_i, D_i) \in \scr Y$.
\end{es}

\begin{lemma}\label{lemma_card}
  Let $\scr A$ be a rank $d$ toric arrangement. For all $i=0,\ldots d$, we have $\vert \scr Y_i \vert = \vert \scr N_i\vert$.
\end{lemma}

\begin{proof}
  This follows because for every $i=0,\ldots,d$, 
  \begin{displaymath}
    \vert \scr N_i\vert = \sum_{\substack{Y\in \call C(\scr A)\\\dim Y
      = i}} \vert \nbc_i(\scr A[Y]) \vert
  \end{displaymath}
  Every summand on the right hand side counts the number of generators
  in top degree cohomology or - equivalently - the number of top
  dimesional cells of a minimal CW-model of the complement of the complexification
  of $\scr A [Y]$. By \cite[Lemma 4.18 and Proposition 2]{delucchi} these top
  dimensional cells correspond bijectively to chambers $C\in \call
  T(\scr A[Y])$ with $X(Y,C)=Y$. Therefore 
  \begin{displaymath}
    \vert \scr N_i\vert = \sum_{\substack{Y\in \call C(\scr A)\\\dim Y
      = i}} \vert \{C\in \call T(\scr A[Y]) \mid X(Y,C) = Y \} \vert =
  \vert \scr Y_i\vert.
  \end{displaymath}
\end{proof}

\begin{definition}\label{def:total}
  Recall Definition \ref{def:mu} and define a function 
  \begin{eqnarray*}
    \xi_0:& \scr Y &\to \call T(\scr A_0)_B\\
    &(Y,C) &\mapsto\mu[\scr A[Y],\scr A_0](C)
  \end{eqnarray*}
 Choose, and fix, a total
  order $\dashv$ on $\scr Y$ that makes this function order preserving
  (i.e., for $y_1,y_2\in \scr Y$, by definition 
  $\xi_0(y_1)\prec_0 \xi_0(y_2)$ implies $y_1\dashv y_2$).
\end{definition}

We now examine the local properties of the ordering $\dashv$.

\begin{definition}
  For $F\in \call F(\scr A)$ let $\scr Y_{F}:=\{(Y,C)\in \scr Y \mid
  F\subseteq Y \}$. 

Since $F\subseteq Y$ implies $\scr A [Y] \subseteq \scr A [F]$, we can
define a function 
\begin{eqnarray*}
  \xi_F:& \scr Y_F &\to \call
  T(\scr A[F])\\
 &(Y,C)&\mapsto \mu[\scr A[Y],\scr A[F]](C)
\end{eqnarray*}
\end{definition}

\begin{oss} By Lemma \ref{lemma_mu}, $\mu[\scr A[F],\scr A_0] \circ \xi_F =
  \xi_0$ on $\scr Y_F$. Therefore, for $y_1,y_2\in \scr Y_F$, 
  $\xi_F(y_1)\prec_F \xi_F(y_2)$ implies $\xi_0(y_1)\prec_0
  \xi_0(y_2)$, and thus  $y_1\dashv y_2$.
\end{oss}

\begin{prop}\label{propo_x}
  For all $F\in \call F(\scr A)$ and every $y=(Y,C)\in \scr Y_F$, 
  \begin{displaymath}
    X(F,\xi_F(y)) = Y.
  \end{displaymath}
\end{prop}
\begin{proof}  We will use the lattice isomorphisms $\call L(\scr
  A[F])_{\leq \tilde{Y}}\simeq \call L(\scr A[Y]) \simeq \call C(\scr
  A)_{\leq Y}$. 
  By definition we have that 
  \begin{displaymath}
    \xi_F(y)=\mu[\scr A[Y],\scr A[F]](C)=\min_{\prec_F}\{K\in \call
    T(\scr A[F])\mid
    K\subseteq C\}
  \end{displaymath}
  and therefore 
  $\scr A [F]_{\widetilde{Y}} \cap S_F(\xi_F(y),C_1) \neq \emptyset $ for
   all $
  C_1\prec_F \xi_F(y)$,
  which shows that $\widetilde{Y} \geq \widetilde{X}(F,\xi_F(y))$ in
  $\call L (\scr A [F])$ and thus
  $Y \geq X(F,\xi_F(y))$ in $\call C(\scr A)$. 
  Now, for every layer $Z$ with
  $Z<Y$ we have that $\scr A[Z]\subseteq \scr A[Y]$. Because by
  definition $Y=X(Y,C)$, we have $\tilde{Z}<\tilde{Y}=\tilde{X}(Y,C)$
  in $\call L(\scr A[Y])$ and so there is $C_2\prec_Y C$ with 
  $S_Y(C_2,C)\cap A [Y]_{\widetilde{Z}} = \emptyset$.

  Let $C_3:=\mu[\scr A[Y],\scr A[F]](C_2)$. We have $C_3\subseteq C_2$ and $\xi_F(y)\subseteq C$, therefore $S_F(C_3,\xi_F(y))\cap
  \operatorname{supp}(\widetilde{Z})=\emptyset$, and $C_3\prec_F \xi_F(y)$ by
  $C_2\prec_Y C$. This means $Z\not\geq X(F,\xi_F(y))$, and the claim follows.
\end{proof}

\begin{lemma}\label{lemma:xibijective}
  For $F \in \call F(\scr A)$ and $C \in \call T(\scr A[F])$ we have
  \begin{displaymath}
    \xi_F(X_C, \sigma_{\scr A[X_C]}(C)) = C
  \end{displaymath}
  In particular $\xi_F: \scr Y_F \to \call T(\scr A[F])$ is a
  bijection.
\end{lemma}

\begin{proof}
  Using the definition of $\xi_F$ and Corollary \ref{cor:minimalchamber} we have
  \begin{displaymath}
    \xi_F(X_C, \sigma_{\scr A[X_C]}(C)) = 
    \mu[\scr A[X_C], \scr A[F]](\sigma_{\scr A[X_C]}(C)) 
   \end{displaymath}
   \begin{displaymath}
    = 
    \min \{K \in \call T(\scr A[F]) \mid  K_{X_C} = C_{X_C} \} = C.
  \end{displaymath}
Letting $\beta_F: \call T(\scr A[F]) \to \scr Y_F $ be defined by
$C\mapsto (X_C,\sigma_{\scr A[X_C]}(C))$, the above means
$\xi_F\circ\beta_F = id$, 
  therefore the map $\xi_F$ is surjective. Injectivity of $\xi_F$
  amounts now to proving 
    $\beta_F\circ\xi_F = id$, which is an easy check of the definitions.
\end{proof}

\begin{cor}\label{cor:xibij}
  For $y_1,y_2\in\scr Y_F$, $y_1\dashv y_2$ if and only
  if $\xi_F(y_1)\preceq_F\xi_F(y_2) $.
\end{cor}

\subsection{Lifting faces and morphisms}\label{sec:liftingFm}

We now relate our constructions to the
covering $\sopra{\scr A}$ of $\scr A$ of
\S \ref{sec:covering}. Recall that $\Lambda$ acts freely on $\call
F(\sopra{\scr A})$ and that $q:\call F(\sopra{\scr A}) \to \call
F(\scr A)=\call F(\sopra{\scr A}) / \Lambda$ is the projection to the
quotient (compare Proposition \ref{prop:quoziente}).

\begin{oss}\label{rem:translate}
  Fix a face $F \in \ob \call F(\scr A)$, and choose a lifting $\sopra
  F $ in $\call F(\sopra{\scr A})$. Then the arrangements $\sopra{\scr
    A}_{\sopra F}$ and $\scr A[F]$ differ only by a translation.
Thus we have natural isomorphisms of posets $$\call F(\scr A[F]) \simeq \call F(\sopra{\scr
    A}_{\sopra F}) \simeq \call F(\sopra{\scr
  A})_{\geq \sopra F} .$$ In the following
  we will identify these posets and, in particular, define a functor
  of acyclic categories $q:\call
  F(\scr A[F]) \to \call F(\scr A)$ according to the
  restriction of $q: \call F(\sopra{\scr A})\to \call F (\scr A)$ to $\call F(\sopra{\scr
  A})_{\geq \sopra F}$.

  Given a face $G$ of $\call
  F(\scr A [F])$ we will write $q(G)$ for the image under the covering
  $q$ (see Proposition \ref{prop:quoziente}) of the
  corresponding face of $\call F(\sopra{\scr
  A})_{\geq \sopra F}$ .
\end{oss}

\begin{oss}[Notation]\label{oss:abuse}
  Recall that we identify posets (such as $\call F(\sopra
  {\scr A})$ or $\call F (\scr A[F])$) with the associated acyclic
  categories, as explained in Remark \ref{acicats:posets}. In
  particular, if $x,y$ are elements in a poset with $x\leq y$, we will
  take the notation $x\leq y$ also to stand for the unique morphism $x\to y$ in the associated category.
\end{oss}

Now, given a morphism $m: F\to G$ of $\call F(\scr A)$, for every
choice of a $\sopra F \in \call F(\sopra{\scr A})$ lifting $F$, there
is a unique morphism $\sopra F \leq  \sopra G$ lifting $m$. We
have $\call F(\sopra{\scr A}_{\sopra G}) \subseteq \call F(\sopra{\scr
  A}_{\sopra F}) $ (see Remark \ref{rem:Rsubarr})

\begin{definition}\label{def:Fm}
  Consider a toric arrangement $\scr A$ on $T_\Lambda \cong (\mat C^*)^k$ and a
  morphism $m: F \to G$ of $\call F (\scr A)$. Because of the freeness
  of the action of $\Lambda$, for every
choice of a $\sopra F \in \call F(\sopra{\scr A})$ lifting $F$, there
is a unique morphism $\sopra F \leq  \sopra G$ lifting $m$.

To $m$ we associate
\begin{itemize}\item[(a)] the order preserving function 
  \begin{displaymath}
    i_m: \call F(\scr A[G]) \to \call F (\scr A[F])
  \end{displaymath}

corresponding to the inclusion $\call F(\sopra{\scr
    A}_{\sopra G}) \subseteq \call F(\sopra{\scr A}_{\sopra F}) $ (see
  Remark \ref{rem:Rsubarr}) under the identification of Remark
  \ref{rem:translate}.
\item[(b)] the face $F_m\in \call F(\scr A[F])$ defined by 
  \begin{displaymath}
    F_m:= i_m(\widehat{G})
  \end{displaymath}
where $\widehat{G}$ denotes the unique minimal element  of $\call
F(\scr A[G])$. 

Clearly then $\widehat{G}=F_{\id_G}$. In the following we
will abuse notation for the sake of transparency and, given a face
$G$ of $\call F(\scr A)$, we will write
$G_{\id}$ for $F_{\id_G}$. 
\end{itemize}

\end{definition}

\begin{figure}[tp]
  \begin{tikzpicture}
    \tikzstyle{every path} = [draw, line width = 1.5pt]
    
    \path 	(0,0) -- (2,2);
    \path [red]	(2,2) -- (4,4);
    \path 	(0,4) -- (4,0);

    \path [red] 	(5,0) -- (9,4);

    \path [red]	(10,0) -- (12,2);
    \path 	(12,2) -- (14,4);
    \path 	(10,4) -- (14,0);
    \path 	(12,0) -- (12,4);
    
    \begin{scope}
      \tikzstyle{every path} = [fill, opacity = .2]
      
      \path (2,2) -- (4,4) -- (0,4) -- cycle;

      \path (5,0) -- (9,4) -- (5,4) -- cycle;

      \path (10,0) -- (12,2) -- (10,4) -- cycle;
    \end{scope}

    \begin{scope}
      \tikzstyle{every node} = [fill = none]
      
      \path node [anchor = north west] (Fm) at (3,3) {$F_m$}
	    node [anchor = north west] (F) at (7,2) {$F$}
	    node [anchor = north west] (Fn) at (11,1) {$F_n$};
      
      \path node (imC) at (2,3.3) {$i_m(C)$}
	    node (C) at (6.3, 2.7) {$C$}
	    node (inC) at (11, 2) {$i_n(C)$};
      
      \path node (im) at (2.5,5) {$i_m$}
	    node (in) at (10,-1) {$i_n$};
    \end{scope}

    \begin{scope}
      \tikzstyle{every path} = [draw, line width = 1.5pt]
      
      \draw[->] (7.5,3.5) .. controls (4,5) and (2,5) .. (2,3.8);
      \draw[->] (5.5,1) .. controls (9,-1) and (11,-1) .. (10.2,.5);
    \end{scope}
  \end{tikzpicture} 
  \caption{$F_m$ and the map $i_m$}
  \label{fig:Fmim}
\end{figure}
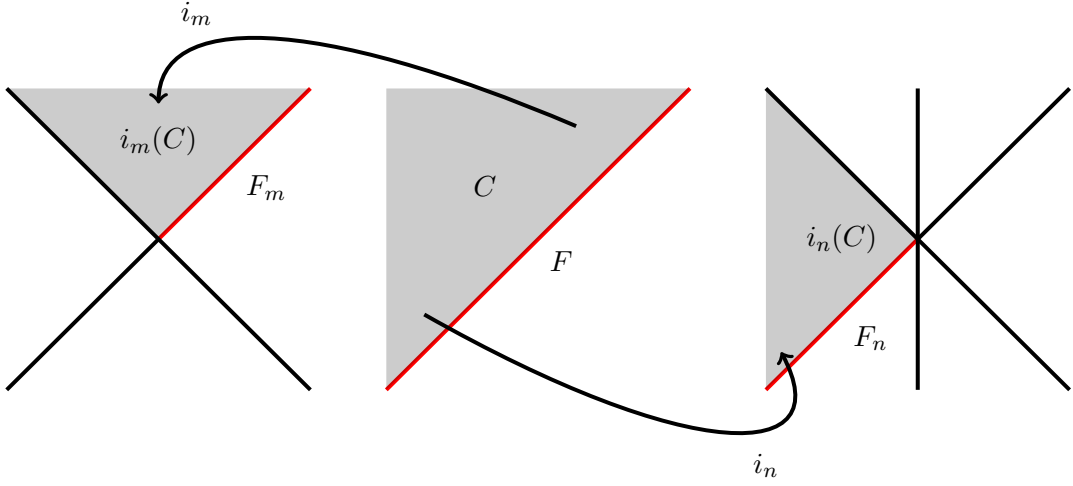

\begin{es}
  Consider the arrangement $\scr A$ of Figure \ref{fig:toricarr}.
  Figure \ref{fig:Fmim} illustrates the maps $i_m$ and $i_n$ 
  for the morphisms $m: P \to F$ and $n: Q \to F$.
\end{es}

\begin{oss}\label{oss:vettorisegnati}
  Every choice of positive sides for the elements of $\scr A_0$
  determines a corresponding choice for all the elements of $\sopra{\scr A}$. 
  Then given $m:F\to G$ and any lift $\sopra G$ of $G$, in terms of sign vectors and identifying each $H \in \scr
  A[F]$ with its unique translate which contains $\sopra G$:
  \begin{displaymath}
    \gamma_{F_m}[\scr A[F]] = {\gamma_{\sopra G}[\sopra{\scr A}]}_{|\scr A[F]}.
  \end{displaymath}
  In particular, when $G$ is a chamber, then $F_m$ also is.
\end{oss}

\begin{lemma} Recall the setup of Definition \ref{def:Fm}.
  \begin{enumerate}[label={(\alph*)},ref={\thecor.(\alph*)}]
\item \label{lemma:imFn}
  If $F\stackrel{m}{\to} G \stackrel{n}{\to} {K}$ are morphisms of
    $\call F(\scr A)$, then
    \begin{displaymath}
      i_{n\circ m} = i_m \circ i_n \quad\textrm{thus}\quad
      i_m(F_n)=F_{n\circ m}.
    \end{displaymath}
\item\label{lem:sisamai} Let $m:F\to G$ be a morphism of $\call F(\scr
  A)$. Then, for every morphism $n$ of
  $\scr A[G]$, we have $q(i_m(n))=q(n)$ and, in particular, for every
  face $K$ of $\scr A [G]$, $q(i_m(K))=q(K)$.
\item \label{oss:upperbound}
  Let $m:G\leq K$ be a morphism of $\call F(\scr A [F])$. Then there
  are 
  morphisms $n: F\to q(G)$ and $m^\downharpoonright$ of $\call F(\scr A)$ with

  \begin{displaymath}
    i_{n} ({q(G)}_{\id} \leq F_{m^\downharpoonright} ) = m
  \end{displaymath}
   \end{enumerate}
\end{lemma}
\begin{proof}
  Parts (a) and (b) are immediate rephrasing of the definitions. 
  For part (c) let $n:= q(F_{\id}\leq G) $ and $m^\downharpoonright :=
  q(m)$.
\end{proof}

\subsection{Definition of the strata}

Each stratum will be associated to an element of $\scr Y$, and we will
think of the Salvetti category as being `built up' from strata
according to the ordering of $\scr Y$.

\begin{definition}\label{def:theta}
Define the map $\theta: \Sal (\scr A) \to \scr Y$ as follows
\begin{displaymath}
  \theta: (m: F \to C) \mapsto ({X(F,F_m)},\, \sigma_{\scr A[X(F,F_m)]}(F_m))
\end{displaymath}  
\end{definition}

\begin{oss}\label{oss:xithetam}
  For every object $m:F\to C$ of $\Sal(\scr A)$  we have $\xi_F(\theta(m))=F_m$.
\end{oss}

\begin{lemma}
For $m:G\to C, m':G\to C'\in \zeta$, if $\theta(m)\dashv \theta(m')$
then $F_m \prec_G F_{m'}$.  
\end{lemma}

\begin{proof}  
  If $\theta(m)\dashv \theta (m')$, then 
  with Remark
  \ref{oss:xithetam} and Corollary \ref{cor:xibij}, 
  $F_m=\xi_G(\theta(m)) \prec_G \xi_G(\theta(m'))=F_{m'}$.
\end{proof}

\begin{definition}\label{def:N}
  Given a complexified toric arrangement $\scr A$ on $(\mat C^*)^d$, we consider the following
  stratification of $\Sal(\scr A)$ indexed by $\scr Y$: 
  $\Sal(\scr A) = \cup_{(Y, C) \in \scr Y} \call S_{(Y,C)}$ where
  \begin{displaymath}
    \call S_{(Y, C)} = \{ m \in \mbox{Sal}(\scr A)\mid
    \exists (m \to n) \in \mbox{Mor}(\mbox{Sal}(\scr A)),\,
    n \in \theta^{-1}(Y, C)\}.
  \end{displaymath}
  Moreover, recall from Definition \ref{def:total} the total ordering $\vdash$ on $\scr Y$ and define
  \begin{displaymath}
    \call N_{y} = \call S_{y} \backslash \bigcup_{y' \dashv y} \call S_{y'}.
  \end{displaymath}
\end{definition}

\begin{figure}[tp]
  \centering

  \subfigure[Stratification of the toric salvetti complex] {
    \begin{tikzpicture}

      \begin{scope}
	\tikzstyle{every node} = [fill, circle, inner sep = 2pt, outer sep = 2 pt]
	
	\path node (C0)	at (   2,   0) {}
	      node (C1)	at (   1,   2) {}
	      node (C2)	at (   2,   4) {}
	      node (C3)	at (   3,   2) {}
	      node (C4)	at (   3,   6) {}
	      node (C5)	at (   5,   6) {}
	      node (C6)	at (   6,   4) {}
	      node (C7)	at (   5,   2) {};
      \end{scope}

      \begin{scope}
	\tikzstyle{every path} = [draw, line width = 2pt]
	
	\path[->] (C0) to [bend left = 30] (C1);
	\path[->] (C3) to [bend left = 30] (C2);
	\path[->] (C6) to [bend left = 30] (C5);
      \end{scope}

      \begin{scope}
	\tikzstyle{every path} = [draw, line width = 10pt, white]  
	\path ( 0,-1) -- ( 0, 5)
	      ( 4,-1) -- ( 4, 9)
	      (-1,-1) -- ( 7, 7)
	      (-1, 3) -- ( 5, 9)
	      ( 3,-1) -- ( 7, 3)
	      (-1, 1) -- ( 1,-1)
	      (-1, 5) -- ( 5,-1)
	      ( 1, 7) -- ( 7, 1)
	      ( 3, 9) -- ( 7, 5);
      \end{scope}

      \begin{scope}
	\tikzstyle{every path} = [draw, line width = 1.5pt]  
	\path ( 0,-1) -- ( 0, 5)
	      ( 4,-1) -- ( 4, 9)
	      (-1,-1) -- ( 7, 7)
	      (-1, 3) -- ( 5, 9)
	      ( 3,-1) -- ( 7, 3)
	      (-1, 1) -- ( 1,-1)
	      (-1, 5) -- ( 5,-1)
	      ( 1, 7) -- ( 7, 1)
	      ( 3, 9) -- ( 7, 5);
      \end{scope}

      \begin{scope}
	\tikzstyle{every path} = [draw, line width = 10pt, white]
	
	\path (C1) to [bend left = 30] (C0);
	\path (C2) to [bend left = 30] (C3);
	\path (C5) to [bend left = 30] (C6);
	\path (C1) to [bend left = 30] (C2);
	\path (C0) to [bend left = 30] (C3);
	\path (C2) to [bend left = 30] (C4);
	\path (C4) to [bend left = 30] (C5);
	\path (C3) to [bend left = 30] (C7);
	\path (C7) to [bend left = 30] (C6);
      \end{scope}

      \begin{scope}
	\tikzstyle{every path} = [draw, line width = 2pt]
	
	\path[->] (C1) to [bend left = 30] (C2);
	\path[->] (C0) to [bend left = 30] (C3);
	\path[->] (C2) to [bend left = 30] (C4);
	\path[->] (C4) to [bend left = 30] (C5);
	\path[->] (C3) to [bend left = 30] (C7);
	\path[->] (C7) to [bend left = 30] (C6);
      \end{scope}

      \begin{scope}
	\tikzstyle{every path} = [draw, line width = 2pt, green]

	\path[->] (C1) to [bend left = 30] (C0);
	\path[->] (C2) to [bend left = 30] (C3);
	\path[->] (C5) to [bend left = 30] (C6);
      \end{scope}

      \begin{scope}
	\tikzstyle{every path} = [draw, blue, line width = 5pt, opacity = .3]
	
	\path (0,0) -- (4,0) -- (4,4) -- (0,4) -- cycle;
      \end{scope}

      \path[pattern = dots, draw = none] (2,0)
	  to [bend left = 30]	( 1, 2)
	  to [bend left = 30]	( 2, 4)
	  to [bend left = 30]	( 3, 6)
	  to [bend left = 30]	( 5, 6)
	  to [bend right = 30]	( 6, 4)
	  to [bend right = 30]	( 5, 2)
	  to [bend right = 30]	( 3, 2)
	  to [bend right = 30]	( 2, 0);

      \path[fill, draw = none, opacity=.4, green] (2,0)
	  to [bend right = 30]	( 1, 2)
	  to [bend left = 30]	( 2, 4)
	  to [bend left = 30]	( 3, 6)
	  to [bend left = 30]	( 5, 6)
	  to [bend left = 30]	( 6, 4)
	  to [bend right = 30]	( 5, 2)
	  to [bend right = 30]	( 3, 2)
	  to [bend right = 30]	( 2, 0);    
    \end{tikzpicture}
  }\hfill
  \subfigure[$\call N_{(K_2,D_2)} = \call F(\scr A^{K_2})^{op} = \call F(\scr A^{K_2})$]{
    \begin{tikzpicture}[scale=1.5]
      \tikzstyle{every path} = [draw, line width = 1.5pt]
      \tikzstyle{every node} = [fill, circle, inner sep = 1.5pt, outer sep = 3pt]
      
      \path node (A) at (0,0) {}
	    node (B) at (2,0) {}
	    node (C) at (0,2) {}
	    node (D) at (2,2) {};
      
      \path [->] (A) to [bend left = 15] (C);
      \path [->] (A) to [bend left = 15] (D);
      \path [->] (B) to [bend right = 15] (C);
      \path [->] (B) to [bend right = 15] (D);
    \end{tikzpicture}
  }
\caption{Stratification of the toric Salvetti Complex}
\label{fig:toric_stratification}
\end{figure}

\begin{es}
  Consider the toric arrangement $\scr A$ of Figure \ref{fig:toricarr}.
  Figure \ref{fig:toric_stratification} (a) shows two strata of
  the stratification on $\Sal \scr A$ of Definition \ref{def:N}.
  
  The stratum $\call S_{((\mat C^*)^2, D)}$ is pictured in dotted black,
  while the startum $\call N_{(K_2, D_2)}$ is pictured in solid green.
  Thus $\call N_{(K_2, D_2)}$ consists of two $1$-dimensional layers
  and two $2$-dimensional layers. The category $\call N_{(K_2, D_2)}$ is
  showed in Figure \ref{fig:toric_stratification} (a) and it is isomorphic
  to $\call F(\scr A^{K_2})$ (which is self-dual).
\end{es}
\section{The topology of the Strata}\label{sec:topstrata}

We now want to show that, for $y\in \scr Y$, the category $\call N_{y}$ is isomorphic
to the face category of a complexified toric arrangement. The main result of this
section is the following.

\begin{teo}\label{teo:pezzinuovi}
  Consider a complexified toric arrangement $\scr A$ and for $y=(Y,C)\in \scr Y$ let $\call N_{y}$ be as in Definition \ref{def:N}. Then there is an isomorphism
  of acyclic categories
  \begin{displaymath}
    \call N_{(Y,C)} \cong \call F(\scr A^Y)^{op}
  \end{displaymath}
\end{teo}

The main idea for proving this theorem is to use the `local' combinatorics of 
the (hyperplane) arrangements $\scr A [F]$ to understand the `global' structure 
of the strata in $\Sal(\scr A)$. We carry out this `local-to-global' approach by using the
language of diagrams.

\subsection{The category $\acicat$} \label{sec:AC}

Let $\cat$ denote the category of small categories. We define
$\acicat$ to be the full subcategory of $\cat$ consisting of acyclic
categories (see Definition \ref{def:acyclic}, compare
\cite{kozlov2007combinatorial}). 


Colimits in $\acicat$ do not coincide with colimits taken in
$\cat$. In the following, we will need
an explicit description of colimits in $\acicat$, at least for the
special class of diagrams with which we will be concerned.

\begin{definition}\label{def:geomcolim}
  Let $\call I$ be an acyclic category. A diagram $\scr D: \call I \to
  \acicat$ of acyclic categories is called {\em geometric} if 
  \begin{itemize}
  \item[(i)]
    \begin{list}{-}{}
    \item   for every $X\in \ob(\call I)$, $\scr D (X)$ is ranked and
    \item  for every $f\in \Mor(\call I)$, $\scr D(f)$ is rank-preserving;
    \end{list}

  \item[(ii)] for every $X\in \ob(\call I)$ and every $x\in \Mor(\scr
    D(X))$ there exist
    \begin{list}{-}{}
    \item       $\widehat{X}\in \ob(\call I)$, 
    \item $f\in \Mor_{\call I}(\widehat{X},X)$ and 
    \item $\widehat{x}\in \Mor \scr D(\widehat{X})$ with $\scr D(f)(\widehat{x})=x$
    \end{list}
such that: for every morphism $g\in \Mor_{\call I}(Y,X)$
    and every $y\in \scr D(g) ^{-1} (x)$ there exists a
    morphism $\widehat{g}\in \Mor_{\call I}(\widehat{X},Y)$ such that
    $\scr D(\widehat {g})(\widehat{x})=y$.
  \end{itemize}
\end{definition}

\begin{oss}\label{rem:unihat}
  From the definition it follows that the morphism $\widehat{x}$ in (ii) is unique.
\end{oss}

\begin{definition}\label{def:eqrel:diag}
  Define a relation $\sim$ on $\coprod_{X\in \ob \call I} \Mor(\scr
  D(X))$ as follows: for $x\in \Mor(\scr D(X))$ and $y\in \Mor(\scr
  D(Y))$ let $x\sim y$ if in there is
  \begin{list}{-}{}
  \item an object $Z\in \ob (\call I)$, a morphism $z\in \Mor(\scr
    D(Z))$
   \item morphisms $f_X:Z\to X$, $ f_Y: Z\to Y $ of $\call I$
  \end{list}
such that $\scr D(f_X)(z)=x$ and
  $\scr D(f_Y)(z)=y$.

  Moreover, define a relation $\approx$ on $\coprod_{X\in \ob \call I}
  \ob(\scr D (X))$ by setting $a\approx b$ if $\id_a \sim \id_b$.
\end{definition}

\begin{oss} If $\scr D$ is a geometric diagram of acyclic categories, 
  the observation that $x\sim y$ if and only if
  $\widehat{x}=\widehat{y}$, together with Remark \ref{rem:unihat},
  shows that $\sim$ and $\approx$ are in fact equivalence relations.
\end{oss}

\begin{prop}\label{prop:geomcolim} Let $\scr D:\call I \to \acicat$ be a gometric diagram of acyclic
  categories. Then, 
  the colimit of $\scr D$ exists and is given by the
    co-cone $(\call C ,(\gamma_X)_{X\in \ob\call I})$ with
    \begin{displaymath}
      \ob(\call C)=\coprod_{X\in \ob \call I} \ob(\scr D (X)) \bigg/ \approx, 
      \quad\quad
      \Mor(\call C) = \coprod_{X\in \ob \call I} \Mor(\scr D(X)) \bigg/ \sim
    \end{displaymath}
    (where $[m]_\sim : [x]_\approx \to [y]_\approx$ whenever $m: x\to
    y$), and for every $X\in\ob\call I$, $x\in \ob\scr D(X)$, $m\in
    \Mor \scr D (X)$:
    \begin{displaymath}
      \gamma_X(x)=[x]_\approx,\quad\quad \gamma_X(m)=[m]_\sim .
    \end{displaymath}
\end{prop}

\begin{proof} One easily checks that $\call C$ is
  a well-defined small category. We have to prove two claims.

\begin{list}{}{\leftmargin=0.5em}
  \item[{\em Claim 1: $\call C$ is acyclic.}]$\,$ 

   {\em Proof:} Because the definition of a geometric diagram requires
   $\scr D(f)$ to be rank-preserving for all $f\in \Mor \call I$, we
   can define for all $[x]_\approx\in \ob \call C$ a value
   $\nu([x]_\approx):=\rk(x)$, where $x$ is any representant and the
   rank is taken in the appropriate category. Now, for every $X\in
   \ob \call I$, every nonidentity morphism $m\in \Mor_{\scr D
     (X)}(x,y)$ has $\rk(x)<\rk(y)$ and thus $\nu([x]_\approx) <
   \nu([y]_\approx)$ - in particular, $[m]_\sim$ is not an
   identity. This implies directly that the only endomorphisms of
   $\call C$ are the identities. Moreover, if the morphism
   $[m]_\sim $ above is an invertible non-identity, then its
   inverse whould be a morphism $[y]_\approx \to [x]_\approx$ - but
   since $\nu([x]_\approx) < \nu([y]_\approx)$, no such morphism exists.

   \item[{\em Claim 2: The co-cone $(\call C, (\gamma_X)_{X\in \ob
     \call I})$ satisfies the universal property.}]$\,$
   
 {\em Proof:} Let $(\call C', (\gamma'_X)_{X\in \ob \call
     I})$ be a co-cone over $\scr D$. We have to show that there is a
   unique morphism of co-cones $\Psi: (\call C, (\gamma_X)_{X\in \ob \call
     I})\to (\call C', (\gamma'_X)_{X\in \ob \call
     I}) $. 

   In order to do that, notice that if $y\in [x]_\sim \in \Mor \call
   C$,  there are $X,Y,Z\in
   \ob\call I$, $f_X,f_Z\in \Mor \call I$ and $z\in \Mor\scr D
   (Z)$ as in  Definition \ref{def:eqrel:diag}, such that
   \begin{displaymath}
     \gamma'_X(x) = \gamma'_Z f_X (z) =\gamma'_Y f_Y(z) =
     \gamma'_Y(y).
   \end{displaymath}
   This proves that the assignments
   \begin{displaymath}
     \Psi[x]_\approx := \gamma'_X(x), \quad \Psi[m]_\sim :=\gamma'_X(m),
   \end{displaymath}
   where $X$ is such that $x$ is in $\scr D(X)$, do not depend on the
   choice of the representant $x$ and thus define a function $\Psi:
   \call C \to \call C'$. A routine check shows functoriality and
   uniqueness of $\Psi$. 
   \end{list}
\end{proof}

\subsection{Proof of Theorem \ref{teo:pezzinuovi}}
Throughout this section let $\scr A$ be a complexified toric
arrangement and recall the notational conventions of Section
\ref{sec:liftingFm}, in particular Remark \ref{rem:translate} and
Remark \ref{oss:abuse}.

\begin{definition}[A diagram for the face category of the compact torus]\label{def:diag:F}
  \begin{align*}
    \scr F(\scr A) := \scr F : \call F(\scr A)^{op} & \to \acicat\\
    F&\mapsto \call
    F(\scr A[F])\\  
    (m: F\to G)& \mapsto 
    (i_m: \call F(\scr A[G]) \to \call F(\scr A[F]))
  \end{align*}
\end{definition}

After these preparations, we turn to the diagrams.

\begin{lemma}\label{colim:F} For the diagram $\scr F $ of Definition
  \ref{def:diag:F} we have
  \begin{displaymath}
    \colim \scr F (\scr A) = \call F(\scr A).
  \end{displaymath}
\end{lemma}

\begin{proof}

  We begin by noticing that $\scr F$ is a geometric diagram. Indeed,
  for a morphism $m: G\leq K$ of $\scr F (F)$ let $n$ and $m^\downharpoonright$, be obtained as in Lemma \ref{oss:upperbound}. Then
  \begin{equation}\label{eq:Fgeom}
  \widehat{F}:=q(G),\quad f:=n^{op},\quad \widehat{m}:= (
  q(G)_{\id}\leq F_{m^\downharpoonright}) .
\end{equation}
satisfy the requirements of Definition \ref{def:geomcolim}.

Accordingly, the objects and morphisms of $\colim \scr F$ are given as
in Proposition \ref{prop:geomcolim}, with the relation $\sim$ generated by $n \sim \scr F(m)(n)$ for every morphism $m: F \to G$ of $\call F(\scr A)$
  and every morphism $n: G' \to G^{''}$ of $\call F(\scr A[G])$ and,
  accordingly, the relation $\approx$ generated by 
  $G' \approx \scr F(m)(G')$ for all morphisms 
 $(m: F \to G) \in\Mor( \call F(\scr A))$ and all    
 $G' \in \Obj(\call F(\scr A[G]))$. For the sake of notational transparency we will omit explicit
  reference to $\sim$ and $\approx$ and denote equivalence classes
  with respect to these equivalence relations simply by $\SX \,\cdot \,\DX$, to avoid confusion with the
  square brackets used to identify elements of the Salvetti
  complex. 

  We prove the Lemma by constructing an isomorphism $\Phi: \call
  F(\scr A) \to \colim \scr F$ as follows.
      For every object $F \in \call F(\scr A)$ define $\Phi(F) := \SX
    F_{id}\DX$, (recall from Definition \ref{def:Fm} that $F_{id}$ is a face
    in $\call F(\scr A[F])$), for every morphism $m: F \to G$ in
    $\call F(\scr A)$ define
    \begin{displaymath}
      \Phi(m) =: \SX F_{id} \leq F_m \DX.\end{displaymath}

  The bijectivity of $\Phi$ is easily seen, so we only need to show
  the functoriality of $\Phi$. To this end consider two composable morphisms
  $F \stackrel{m}{\to} G \stackrel{n}{\to} H$. Using Lemma \ref{lemma:imFn} we get
  \begin{align*}
    \Phi(n)\circ\Phi(m) = \SX G_{id} \leq G_n\DX \circ \SX F_{id} \leq
    F_m \DX & \\
    = \SX\scr F(m)(G_{id} \leq G_n)\DX \circ \SX F_{id} \leq F_m\DX &= 
    \SX i_m(G_{id}) \leq i_m(G_n) \DX \circ \SX F_{id} \leq F_m\DX  \\
    = \SX F_m \leq F_{n\circ m} \DX \circ \SX F \leq F_m\DX &=    
    \SX F \leq F_{n\circ m} \DX = \Phi(n\circ m).
  \end{align*}
\end{proof}

\begin{definition}[A diagram for the Salvetti category]\label{def:DDD}
  \begin{align*}
    \scr D (\scr A):= \scr D : \call F(\scr A )^{op} &\to \acicat;\\ 
    F &\mapsto \Sal(\call A[F]); \\  
    (m:F\to G) &\mapsto j_m:\Sal(\scr A[G]) \hookrightarrow
    \Sal(\scr A[F])
  \end{align*}
  where $j_m([G, C]) = [i_m(G), i_m(C)]$.
\end{definition}

\begin{lemma}\label{lemma:colimsal}
  \begin{displaymath}
    \colim \scr D(\scr A) = \Sal (\scr A)
  \end{displaymath}
\end{lemma}

\begin{proof}[Proof of Lemma \ref{lemma:colimsal}]
The proof follows the outline of the proof of Lemma \ref{colim:F}, and
starts by noticing that the diagram $\scr D$, too, is
geometric. Indeed, let $x: [K_1,C_1]\leq [K_2,C_2]$ be a morphism in
$\Sal(\scr A[F])$. Correspondingly, we have morphisms 
$m_0:K_2\leq K_1$, $m_1:K_1\leq C_1$, $m_2:K_2\leq C_2$ of $\call F
(\scr A[F])$. For $i=0,1,2$ let $n_i$, $m_i^\downharpoonright$ be
obtained from $m_i$ as in Lemma \ref{oss:upperbound}. Then a straightforward check of the definitions shows
that the assignment
\begin{equation*}
  \widehat{F} := q(K_2), \quad f:=n^{op},\quad \widehat{x}:
  [(K_2)_{\id},
  F_{m_2^\downharpoonright}]\leq[F_{m_0^\downharpoonright},F_{m_1^\downharpoonright
  \circ m_0^\downharpoonright}]
\end{equation*}
is well-defined and satisfies the requirement of Definition \ref{def:geomcolim}.

Thus Proposition \ref{prop:geomcolim} again applies and,
accordingly, we write objects and morphisms of $\colim \scr D$ as
equivalence classes of the appropriate relations, that we will again denote by $\SX\cdot\DX$.

An isomorphism $\Psi: \Sal(\scr A) \to \colim \scr D$ can now be defined as follows.
For an object $m: F \to C$ of $\Sal(\scr A)$ 
define $\Psi(m)  = \SX [F_{id}, F_m ]\DX$ (notice that,
considering $m$ as  a morphism of $\call F(\scr A)$, we have $\Psi(m)=\Phi(m)$).
For a morphism $(n,m_1,m_2)$ of $\Sal(\scr A)$ 
with $m_i: F_i \to C_i$ and $n: F_2 \to F_1$ define
\begin{align*}
  \Psi(n,m_1,m_2) = \left\SX \scr D(n)([(F_1)_{id}, F_{m_1}]) \leq [(F_2)_{id}, F_{m_2}] \right\DX=\\
  \SX [i_n((F_1)_{id}), i_n(F_{m_1})] \leq [(F_2)_{id}, F_{m_2}] \DX=
  \SX [F_n, F_{m_1\circ n}] \leq [(F_2)_{id}, F_{m_2}] \DX,
\end{align*}
where in the last equality we used Lemma \ref{lemma:imFn}.
\end{proof}

\begin{oss}\label{oss:elementi_sal} 
  Using Remark \ref{oss:upperbound} we have that every element 
  $\varepsilon \in \ob(\colim \scr D(\scr A))$ has a (unique)
  representant $[F_{\id},C] \in \call S(\scr A[F])$ 
  such that for every other representant $[G,K]$ with
  $\varepsilon = \SX G',K \DX$ there is a unique face $G\in\call
  F(\scr A)$ and a unique
  morphism
  $m: F \to G$ with $[G',K] = [F_m, i_m(C)]$.
\end{oss}

\begin{lemma}\label{lemma:restricts}
  Let $m: F \to G$ be a morphism of $\call F(\scr A)$ and consider an $(Y, C) \in \scr Y_F$.
  Then the inclusion $j_m: \Sal(\scr A[G]) \to \Sal(\scr A[F])$ restricts to an inclusion
  \begin{displaymath}
   j_m: \call S_{\xi_G(Y,C)} \to \call S_{\xi_F(Y,C)}.
  \end{displaymath}
\end{lemma}
\begin{oss}\label{rem:inclusione}
  Note that, given any chamber $C$ of $\scr A[G]$ and any chamber $C'$
  of $\scr A[F]$, there is a natural inclusion $\call S(\scr A[G])_C \hookrightarrow \call
  S(\scr A[F])_{C'}\subseteq \call S(\scr A[F])$ if and only if
  $S(i_m(C), C')\cap\scr A[G]=\emptyset$.
\end{oss}
\begin{proof}
  With Remark \ref{rem:inclusione} we only need to show that $S(i_m(\xi_G(Y,C)), \xi_F(Y,C)) \cap \scr A[G] = \emptyset$.
  Let $H \in \scr A[G]$, then
  \begin{displaymath}
    \gamma_{i_m(\xi_G(Y,C))}(H) = \gamma_{\xi_G(Y,C)}(H) = \gamma_{\xi_F(Y,C)}(H)
    \Longrightarrow H \notin S(i_m(\xi_G(Y,C)), \xi_F(Y,C))
  \end{displaymath}
  where the last equality follows from the fact that $\xi_F(Y,C) \subseteq \xi_G(Y,C)$. 
\end{proof}

Lemma \ref{lemma:restricts} allows us to state the following definition.
\begin{definition}
Given $(Y,C)\in \scr Y$ let
  \begin{align*}
    \scr E_{(Y,C)}:  \call F(\scr A^Y)^{op} &\to \acicat \\
    F&\mapsto \call S(\scr A[F])_{\xi_F(Y,C)}  \\
    (m: F\to G)   &\mapsto (j_m)_{\vert \scr E_{(Y,C)}(G)}
  \end{align*}
\end{definition}

\begin{lemma}\label{lemma:colimS}
  Let $(Y,C) \in \scr Y$, then
  \begin{displaymath}
    \colim \scr E_{(Y,C)} = \call S_{(Y,C)}
  \end{displaymath}
\end{lemma}
\begin{proof}
  We consider the isomorphism $\Psi: \Sal(\scr A) \to \colim\scr D$ of Lemma
  \ref{lemma:colimsal}. We want to show that $\Psi(\call S_{(Y,C)}) =
  \colim \scr E_{(Y,C)}$, and we do this in two steps.
  
  \begin{list}{}{\leftmargin=0.5em}
  \item[{\em Step 1: $\colim \scr E_{(Y,C)}\subseteq \Psi(\call
      S_{(Y,C)}) $.}] $\,$

    Let $\SX G, K \DX \in \colim \scr E_{(Y,C)}$,
    then (recall Remark \ref{oss:elementi_sal}) there is a morphism of
    $\call F(\scr A)$ $m:F \to G$ such that $[F_m, i_m(K)] \in \call
    S_{\xi_F(Y,C)} \subseteq \Sal (\scr A[F])$, i.e.
    \begin{displaymath}
      [F_m, i_m(K)] \leq [F_{\id}, \xi_F(Y,C)].
    \end{displaymath}
    Taking the preimage through $\Psi$ of this relation we get a
    morphism 
    \begin{displaymath}
      \Psi^{-1}(\SX G,K \DX) \to \Psi^{-1}(\SX F_{\id}, \xi_F(Y,C) \DX)  \in \Mor (\Sal (\scr A)).
    \end{displaymath}
    Now, using Proposition \ref{propo_x} we have
    \begin{displaymath}
      \theta(\Psi^{-1}(\SX F_{\id}, \xi_F(Y,C) \DX)) = (X(F, \xi_F(Y, C)),
      \sigma_{\scr A[Y]} \xi_F(Y,C)) 
    \end{displaymath}
    \begin{displaymath}
      =(Y, \sigma_{\scr A[Y]}\circ \mu[\scr A[Y], \scr A[F]] (C)\,) = (Y, C).
    \end{displaymath}
    Therefore $\Psi^{-1}(\SX G, K \DX) \in \call S_{(Y,C)}$, so $\SX
    G, K \DX \in \Psi\sx \call S_{(Y,C)}\dx$, as was to be proved.
  
  \item[{\em Step 2: $ \Psi(\call S_{(Y,C)})\subseteq \colim \scr
      E_{(Y,C)}$.}] $\,$

    Consider now $(m: G \to K) \in \call
    S_{(Y,C)}$. Then there is a morphism $(h,m,n): m \to n \in
    \Mor(\Sal (\scr A))$ with $n: F \to K'$, $h: F \to G$ and
    $\theta(n) = (Y,C)$. In particular, in view of Remark
    \ref{oss:xithetam}, we get $F_n = \xi_F(\theta(n)) = \xi_F(Y,C)$.

    Applying $\Psi$ to the morphism $(h,m,n)$, in $\Sal (\scr A[F] )$
    we obtain
    \begin{displaymath}
      j_n([G, G_m]) \leq [F, F_n] = [F, \xi_F(Y,C)], \textrm{ thus }
      j_n([G, G_m]) \in \call S_{\xi_F(Y,C)},
    \end{displaymath}
    and we conclude that
    \begin{displaymath}
      \Psi(m) = \SX G, G_m \DX = \SX j_n([G, G_m]) \DX \in \colim \scr E_{(Y,C)},
    \end{displaymath}as required.
  \end{list}

\end{proof}

\begin{definition} Given $(Y,C)\in \scr Y$, define
  \begin{align*}
    \scr G_{(Y,C)}: \call F(\scr A^Y)^{op} &\to \acicat \\ 
    F &\mapsto \call N_{\xi_F(Y,C)} \\
    (m: F\to G) &\mapsto (j_m)_{\vert \scr G_{(Y,C)}(G)}
  \end{align*}
\end{definition}

\begin{oss}
    To prove that the diagram $\scr G_{(Y,C)}$ is well defined, we have
    to show that for every morphism $m: F\to G$ of $\call F(\scr A^Y)$ 
  \begin{equation}\label{eq:incl}
  j_m(\call N_{\xi_G(Y,C)})\subseteq \call N_{\xi_F(Y,C)}.
  \end{equation}
  
 This follows because by Proposition \ref{propo_x} we have 
  $X(F, \xi_F(Y,C)) = Y$, and thus with \cite[Lemma
  4.18]{delucchi} we can rewrite 
  \begin{displaymath}
    \call N_{\xi_F(Y,C)} = \{ [G,K] \in \Sal (\scr A[F]) \mid
	G \in \call F(\scr A[F]^{\tilde Y}),\, K_G = \xi_F(Y,C)_G \}.
  \end{displaymath}
  Now let $[G',C'] \in \call N_{\xi_G(Y,C)}$. Then since $G' \subseteq
  \tilde Y $ we have $
  i_m(G') \in \call F(\scr A[F]^{\tilde Y})$, and from
  $\xi_F(Y,C) \subseteq \xi_G(Y,C)$ we conclude $ i_m(C')_{G'} = \xi_F(Y,C)_{G'}$.
  Therefore $j_m([G', C']) = [i_m(G'), i_m(C')] \in \call
  N_{\xi_F(Y,C)}$, and the inclusion \eqref{eq:incl} is proved.
\end{oss}

\begin{lemma}\label{lemma:colimG}
  \begin{displaymath}
    \colim \scr G_{(Y,C)} = \call N_{(Y,C)}
  \end{displaymath}
\end{lemma}

\begin{proof} Again the proof is in two steps.
  \begin{list}{}{\leftmargin=0.5em}
  \item[{\em Step 1:} $\colim \scr G_{(Y,C)} \subseteq \call
    N_{(Y,C)}$.]$\,$

 For this, let $\SX F,K \DX \in \colim \scr G_{(Y,C)}$
    and suppose $\SX F,K \DX \notin \call N_{(Y,C)}$.  Then $\SX F,K
    \DX \in \colim \scr E_{(Y', C')}$ for some $(Y', C') < (Y, C)$.
    Now, since $\SX F,K \DX \in \colim \scr G_{(Y,C)}$ there exist a
    point $P \in \call F(\scr A)$ and a morphism $m: P \to F$ with
    $[P_m, i_m(K)] \in \call N_{\xi_P(Y,C)}$. Therefore, in $\scr
    A[P]$ we have $[P_m, i_m(K)] \leq [P, \xi_P(Y,C)]$, which implies
    $K_{P_m} = \xi_P(Y, C)_{P_m} $, and thus $ K = \sigma_{\scr
      A[F]}(K_{P_m}) = \xi_F(Y,C)$.

    Similarly, since $\SX F,K \DX \in \colim \scr E_{(Y',C')}$ there
    is a point $Q \in \call F(\scr A)$ and a morphism $n: Q \to F$
    with $[Q_n, i_n(K)] \in \call S_{\xi_Q(Y',C')}$. Then, as above,
    $K = \xi_F(Y',C')$.

    From the bijectivity proven in Lemma \ref{lemma:xibijective} we
    conclude $(Y,C) = (Y', C')$, which contradicts $(Y', C') < (Y,
    C)$, proving that $\SX F,K \DX \in \call N_{(Y,C)}$, as desired.
  
    \item[{\em Step 2: $\call N_{(Y,C)}\subseteq \colim \scr G_{(Y,C)}
        $.}] $\,$

Suppose $[F,K] \in \call N_{(Y,C)}
    \backslash \colim \scr G_{(Y,C)}.$ Then $[F,K] \in \call
    S_{\xi_P(Y',C')}$ for some point $P \in \call F(\scr A)$ and some
    $(Y', C') < (Y, C)$. But then $[F, K] \in \colim \scr E_{(Y',
      C')}$, thus $[F, K] \notin \call N_{(Y,C)}$.
  \end{list}

\end{proof}


  


\begin{lemma}\label{lemma:equivalence}
  There is an equivalence of diagrams
  \begin{displaymath}
    \scr G_{(Y,C)} \cong \scr F(\scr A^Y)^{op}
  \end{displaymath}
\end{lemma}
\begin{proof}
  For each $F \in \call F(\scr A^Y)$ define the isomorphisms $\scr G_{(Y,C)}(F) \to \scr F(\scr A^Y)^{op}(F)$
  as follows
  \begin{displaymath}
    \scr G_{(Y,C)}(F) =\call N_{\xi_F(Y,C)} \cong \call
    F(\scr A[F]^{\tilde Y})^{op} = \call F(\scr A^Y[F])^{op}= \scr F(\scr A^Y)^{op}(F).
  \end{displaymath}
  Where the isomorphism in the middle comes from Theorem \ref{teo:min_delu}.
  
  It can be easily checked that these isomorphisms are indeed morphisms of diagrams.
\end{proof}

As a consequence of Lemma \ref{lemma:equivalence} we can write the following.
\begin{proof}[Proof of Theorem \ref{teo:pezzinuovi}]
  \begin{displaymath}
    \call N_{(Y,C)} = \colim \scr G_{(Y,C)} \cong \colim \scr F(\scr A^Y)^{op}
    = \call F(\scr A^Y)^{op}.
  \end{displaymath}
\end{proof}

\section{Minimality of toric arrangements}
\label{sec:minimality}

In this section we will construct a perfect acyclic matching of the
Salvetti category of a complexified toric arrangement. By Remark
\ref{oss:perfect} this will imply minimality and, with it, torsion
freeness of the arrangement's complement.

\subsection{Perfect matchings for the compact torus}

Let $\scr A$ be a complexified toric arrangement in $T_{\Lambda}$ and
recall the notations of Section \ref{sec:toric_introduction}. 
Choose a point $P\in \max\call C(\scr A)$. Up to a biholomorphic
transformation we may suppose that $P$ is the origin of the torus.

Let then $(\chi_1,a_1),\ldots,
(\chi_d,a_d) \in \scr A$ be such that $\alpha_1,\ldots , \alpha_d$ are
($\mathbb Q$-) linearly independent and $P\in K_i$ for all $i=1,\ldots
,d$. For $i=1,\ldots,d$ let $H^1_i$
denote the hyperplane of $\scr A^\upharpoonright$ lifting $K_i$ at the
origin of $\hom(\Lambda,\mathbb R)\simeq \mathbb R^d$. We identify
for ease of notation $\Lambda \simeq \mathbb Z^d \subseteq \mathbb
R^d$, and in particular think of $\alpha_i$ as the normal vector to
$H_i^1$.

For $j\in [d]$ we consider the rank $j-1$ lattice 

\begin{displaymath}
  \Lambda_j:= \mathbb Z^d \cap \bigcap_{i\geq j } H_i^1.
\end{displaymath}

\begin{lemma}
  There is a basis $u_1,\ldots, u_d$ of $\Lambda$ such that for all
  $i=1,\ldots ,d$, the elements $u_1,\ldots ,u_{i-1}$ are a basis of
  $\Lambda_{i}$.

\end{lemma}
\begin{proof}
  The proof is by repeated application of the Invariant Factor
  Theorem, e.g. \cite[Theorem 16.18]{curtis}, to the free $\mathbb
  Z$-submodule $\Lambda_j$ of $\Lambda_{j-1}$.
\end{proof}

\noindent Let $(H_i^1)^+:=\{x\in \mathbb R^d\mid \langle x,\alpha_i \rangle
\geq 0\}$.

\begin{oss}
  In particular, 
  $u_i\not\in H_i^1$, hence $u_i(H_i^1)\neq H_i^1$. Moreover, without loss of generality we may  suppose $ u_i\in (H_i^1)^+$.
\end{oss}

The lattice $\Lambda$ acts on $\mathbb R^d$ by translations. Given
$u\in \Lambda$, let the corresponding translation be
\begin{displaymath}
  t_u:\mathbb R^d \to \mathbb R^d;\quad\quad x\mapsto t_u(x):= x+ u.
\end{displaymath}

\begin{cor}\label{cor:invariante}
  For all $x\in \mathbb R^d$ and all $ i<j\in [d]$, $\langle
  t_{u_i}(x) , \alpha_{d-j} \rangle = \langle x,\alpha_{d-j} \rangle$.
\end{cor}
\begin{proof} We have $u_i \in \Lambda _j\subseteq H_{d-j}^1$,
  therefore $\langle u_i,\alpha_{d-j}\rangle=0$ and thus
  \begin{displaymath}
    \langle
  t_{u_i}(x) , \alpha_{d-j} \rangle = \langle x+u_i,\alpha_{d-j}
  \rangle = \langle x,\alpha_{d-j} \rangle + \langle u_i,\alpha_{d-j}
  \rangle = \langle x,\alpha_{d-j} \rangle +0.
  \end{displaymath}
\end{proof}

For $i=1,\ldots ,d$ let  $(H_i^2)^+:=t_{u_i}((H_i^1)^+)$, and define




\begin{displaymath}
  Q:= \bigcap_{i=1}^d [ (H_i^1)^+\setminus (H_i^2)^+ ].
\end{displaymath}

\begin{lemma}
  The region $Q$ is a fundamental region for the action of $\Lambda$
  on $\mathbb R^d$.
\end{lemma}
\begin{proof} For $i=1,\ldots,d$, write 
 \begin{displaymath}
  l_i:=\langle u_i,\alpha_i \rangle.
\end{displaymath}
Then, $Q=\{x\in \mathbb R^d \mid 0\leq \langle x , \alpha_i
\rangle <l_i \textrm{ for all }i=1,\ldots,d\}$. It is clear that $Q$
can contain at most one point for each orbit of the action of
$\Lambda$.

Now choose and fix an $x\in \mathbb R^d$. We want to construct an $y\in Q$
such that $x\in y + \Lambda$. 

To this end write $x_0:=x$ and let $\lambda_d := \lfloor \langle x_0 ,
\alpha_d \rangle / l_d \rfloor$. Then let 
\begin{displaymath}
  x_1:= x_0 - \lambda_du_d, \textrm{ thus } 0\leq \langle x_1 , \alpha_d
  \rangle <l_d.
\end{displaymath}
For every $i\in \{1,\ldots d-1\}$ define now recursively $\lambda_{d-i} := \lfloor \langle x_{i} ,
\alpha_{d-i} \rangle / l_{d-i} \rfloor$ and  $x_{i+1}:= x_i - \lambda_{d-i}u_{d-i}$, so that
\begin{displaymath}
  0\leq \langle x_{i+1} , \alpha_{d-i} \rangle <l_{d-i}
\end{displaymath}
and so, by Corollary \ref{cor:invariante}, for every $j< i$:
\begin{displaymath}
  \langle x_{i+1} , \alpha_{d-j} \rangle = 
  \langle t_{u_{d-i}}^{-\lambda_{d-i}} \cdots
  t_{u_{d-j-1}}^{-\lambda_{d-j-1}}(x_{j+1}),\alpha_{d-j}\rangle =
  \langle x_{j+1},\alpha_{d-j} \rangle \in [0,l_{d-j}[.
\end{displaymath}
After $d$ steps, we will have reached $x_d$, with
\begin{displaymath}
  0\leq \langle x_d , \alpha_{i} \rangle <l_i \textrm{ for all } i=1,\ldots,d.
\end{displaymath}
 Hence $y:=x_d\in Q$ is the required point because, putting $u:=\sum_{i=1}^d \lambda_{i}u_i$,
 we have by construction $x_d=t_{-u}(x)$ and so $x=t_{u}(y)\in
 y+\Lambda$.
\end{proof}


\begin{definition}\label{def.FI} Let $\scr A$ be a rank $d$ toric arrangement, and
let $\call B_d$ be the `boolean poset on $d$ elements', i.e.,  the acyclic category of the subsets of $[d]$ with
the inclusion morphisms. Since $\call B_d$ is a poset, the function
$$ \ob (\call F(\scr A)) \to \ob (\call B_d),\quad\quad F\mapsto
\{i\in [d] \mid F\subseteq K_i\},
$$
induces a well defined functor of acyclic categories
$$
\call I: \call F(\scr A) \to \call B_d^{op}.
$$

For every $I\subseteq [d]$ define the category
$$\call F_I :=\call I^{-1}(I).$$
\end{definition}

Our main technical result about the category $\call F_I$ is the following.
\begin{lem}\label{lemma.FI}
For all $I\subseteq [d]$, the subcategory $\call F_I$ is a poset admitting an acyclic
matching with only one critical element (in top rank).
\end{lem}

We postpone the proof of this lemma after some preparatory steps.
Fix $I\subset [d]$, let $k:=\vert I \vert$.

We consider 
$$Q_I:= Q\cap \big(\bigcap_{i\in I} H_i^1 \big) \setminus \bigcup_{j\not\in
I} \big( H_j^1\cup H_j^2 \big).$$
\begin{figure}
\begin{center}
\includegraphics[scale=0.4]{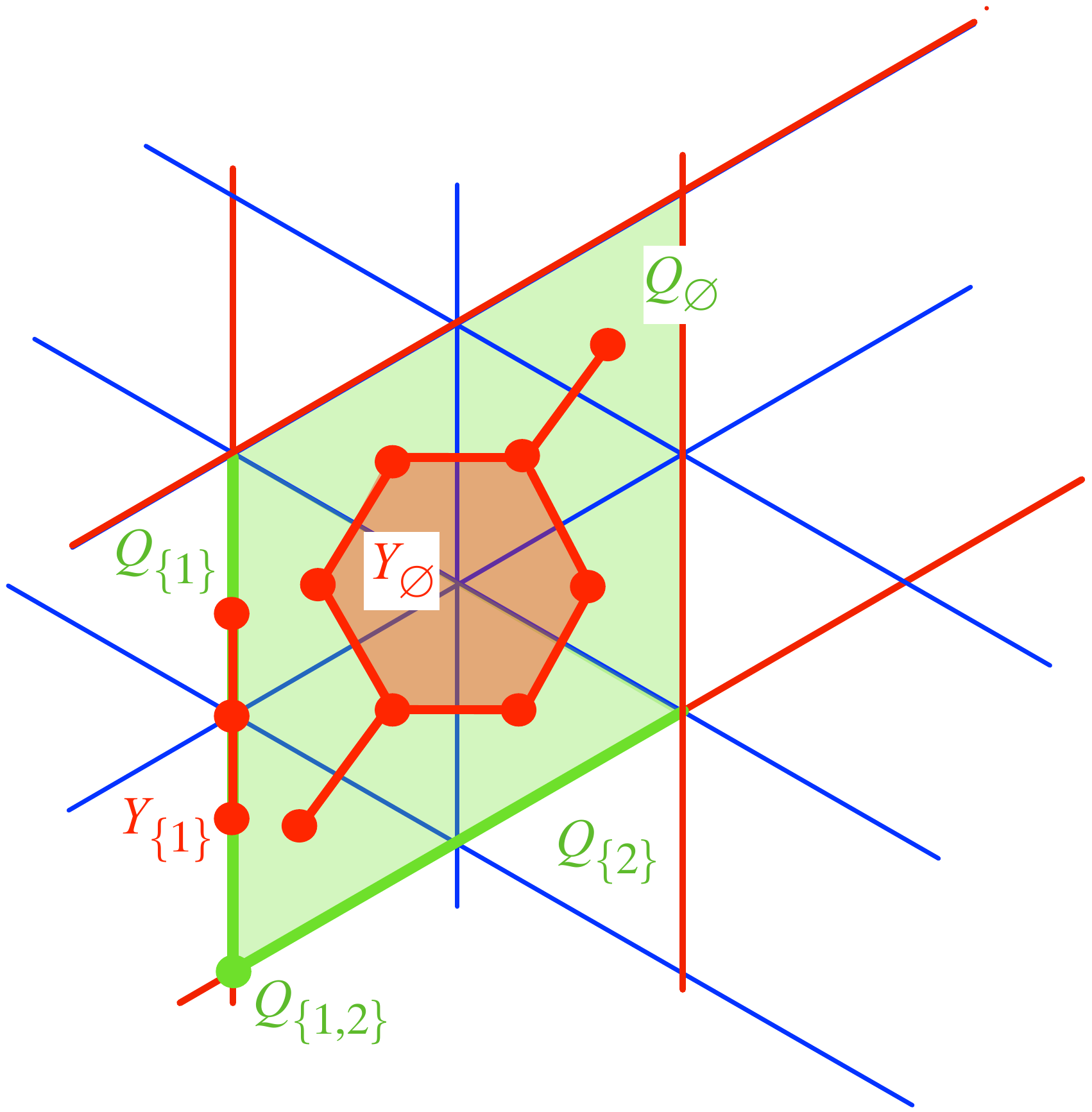}
\caption{The case of the toric Weyl arrangement of Type $A_2$}
\end{center}
\end{figure}

The set $\scr B :=\{H\cap X \mid H\in \scr A^\upharpoonright ,\,
H\cap Q\neq \emptyset \}$ is a finite arrangement of affine hyperplanes
in the affine hull $X$ of $Q_I$. 
This arrangement determines a (regular) polyhedral decomposition $\call D(\scr
B)$ of $\mathbb R^{d-k}$ that coincides with $\call D ({\scr
A^\upharpoonright}_X)$ on $Q$.

The covering
of Section \ref{sec:covering} maps $Q_I$ homeomorphically to its image, hence $\call F_I$ is
the face category of the set of cells of the decomposition of $Q_I$ by
$\call D(\scr B)$. Regularity of $\call D(\scr B)$ implies that $\call
F_I$ is a poset. Indeed, if $\call D(\scr B)^\vee$
is the (regular) CW-decomposition dual to the one induced by $\scr B$,
then $\call F_I^{op}$ is the poset of cells of $Y_I$ (subcomplex of
$\call D (\scr B)$)  that is
entirely contained in $Q_I$.


Let $\call Q$ be the subdivision induced by $\scr B$ on the closure
$\overline{Q_I}$.

\begin{lemma}
  The complex $\call Q$ is shellable.
\end{lemma}

\begin{proof}
Coning the arrangement $\scr B$  (as in \cite[Definition
1.15]{orlik1992ah})  we obtain a central arrangement
$\hat{\scr B}=\{\hat{H} \mid H\in \scr B\}$ which subdivides the unit sphere into a regular cell
complex $\call K$. Then, $\call Q$ is isomorphic to the
subcomplex of $\call K$ given by
\begin{displaymath}
  \bigcap_{i\not\in I} \hat{H_i^1}^+ \cap \bigcap_{i\not\in I} \hat{H_i^2}^-
\end{displaymath}
which, by \cite [Proposition 4.2.6 (c)]{BLSWZ}, is shellable.
\end{proof}

\begin{proof}[Proof of Lemma \ref{lemma.FI}]

The pseudomanifold $\call Q$ is constructible because it is shellable. With
\cite[Theorem 4.1]{Benedetti}, it is also endo-collapsible, i.e., it
admits an acyclic matching where the critical cells are precisely the
cells on the boundary plus one single cell in the interior of $\call
Q$. But this restricts to an acyclic matching of the subposet
$\call F_I \subseteq \call F( \call Q)$ with exacly one critical cell.

In turn this gives an acyclic matching of $\call F_I^{op}$ with
exactly one critical cell. Since $\call F_I^{op}$ is the face poset of
the CW-complex $Y_I$, the critical cell must be in bottom rank - thus
in top rank of $\call F_I$, as required. 
\end{proof}

\begin{prop}\label{prop.minreale}
For any complexified toric arrangement $\scr A$, the acyclic category
$\call F(\scr A)$ admits a perfect acyclic matching.
\end{prop}
\begin{proof}
Let $\scr A$ be of rank $d$. The proof is a straightforward application of the Patchwork Lemma
\ref{PWL} in order to merge the $2^d$ acyclic matchings described in Lemma
\ref{lemma.FI} along the map $\call I$ of Definition  \ref{def.FI}. The resulting
`global' acyclic matching has $2^d$ critical elements and is thus perfect.
\end{proof}

\subsection{Perfect matchings for the toric Salvetti complex}

Let $\scr A$ be a (complexified) toric arrangement. 

\begin{prop}\label{prop:mintorico}
The Salvetti Category $\Sal \scr A$ admits a perfect acyclic matching.
\end{prop}
\begin{proof}  Let the set $\call Y$ be totally ordered according to Definition \ref{def:total}. Let $P$ denote the acyclic category given by the $\vert
  \scr Y \vert$-chain. We define a functor of acyclic
  categories
  $$
  \varphi: \Sal \scr A \to P;\quad m\mapsto (Y,C) \textrm{ for } m \in
  \call N_{(Y,C)}
  $$
  and by Theorem \ref{teo:pezzinuovi} we have an isomorphism of acyclic categories
  $\varphi^{-1}((Y,C))=\call N_{(Y,C)}\simeq \call F(\scr A^{Y})^{op}$. Then, by Proposition \ref{prop.minreale}, $\varphi^{-1}((Y,C))$ has an
  acyclic matching with $2^{d-\rk X}$ critical cells.

An application of the Patchwork Lemma \ref{PWL} yields an acyclic
matching on $\Sal(\scr A)$ with 
$$
\sum_{j} \vert \scr Y_j\vert 2^{d-j}=\sum_{j} \vert \scr N_j\vert 2^{d-j} = P_{\scr A}(1) 
$$
critical cells, where the first equality is given by Lemma \ref{lemma_card}. This matching is thus perfect.
\end{proof}

\begin{cor}
The complement $M(\scr A)$ is a minimal space.
\end{cor}

\begin{proof}
  The cellular collapses given by the acyclic matching of Proposition
  \ref{prop:mintorico} show that the complement $M(\scr A)$ is
  homotopy equivalent to a complex whose cells are counted by the
  Betti numbers.
\end{proof}

\begin{cor}
The homology and cohomology groups 
$H_k(M(\scr A), \mathbb Z)$,
$H^k(M(\scr A), \mathbb Z)$ are torsion free for all $k$.
\end{cor}

\begin{proof}
  See in Corollary \ref{cor:torsion}.
\end{proof}

\section{Application: minimality of affine arrangements}
\label{sec:affine}


After the existence proofs of Dimca and Papadima in \cite{MR2018927} and by Randell in
\cite{MR1900880}, the first step
towards an explicit characterization of the minimal model for
complements of hyperplane arrangements was taken by
Yoshinaga \cite{yoshinaga2007} who, for complexified arrangements, identified the cells of
the minimal complex using their incidence with a general position flag
in real space and studied their boundary maps.
 Salvetti and Settepanella \cite{MR2350466} obtained a complete
 description of the minimal complex by using a `polar ordering'
 determined by a general position flag
 to define a perfect acyclic matching on the Salvetti complex. 

In this section we explain how to use our techniques in order to extend to affine complexified
hyperplane arrangements the idea of \cite{delucchi}. We thus 
obtain a minimal complex that is defined only in terms of the
arrangement's (affine) oriented matroid and is less cumbersome than
the one described in \cite{DeluSette}.



Consider a finite affine complexified arrangement $\scr A = \{K_1, \dots, K_n\}$.
Define the central arrangements $\scr A_0$ and $\scr A[F]$ for
$F \in \call F(\scr A)$ in analogy  to those of Section \ref{sec:strata}.
Choose a base chamber $B \in \call T(\scr A_0)$, fix a total ordering $\prec_0$ on $\scr A_0$
and define $\prec_F, \prec_Y$ for $F \in \call F(\scr A), Y \in \call L(\scr A)$ as in 
Section \ref{sec:strata}. Moreover, let $\scr Y$ be as in Definition \ref{def:ipsilon}.

\begin{oss}
  Notice that, given the affine oriented matroid of $\scr A$, the
  oriented matroid of $\scr A_0$ can be recovered without referring to
  the geometry. For instance, the tope poset of $\scr A_0$ can be
  defined in terms of the tope poset of $\scr A$ based at any
  unbounded chamber.
\end{oss}

\begin{lemma}
  Let $\scr A$ be a finite complexified affine hyperplane arrangement, and $\scr Y$ as above, then
  \begin{displaymath}
    |\scr Y| = \sum_{k \in \mat N} \rk H^k(M(\scr A); \mat Z)
  \end{displaymath}
\end{lemma}
\begin{proof}
  As in Lemma \ref{lemma_card}, applying \cite[Lemma 4.18 and
  Proposition 2]{delucchi}, for all  $Y \in \call L(\scr A)$ we have
  \begin{displaymath}
    |\{C \in \call T(\scr A[Y]) \,|\, X(Y,C) = Y \}| = \rk H^{\codim Y}(M(\scr A_Y); \mat Z).
  \end{displaymath}
  The claim follows with Theorem \ref{teo:brieskorn}.
\end{proof}

We now define the analogue of the map $\theta$ of Definition \ref{def:theta}. 
%
\begin{definition}
  Let $F, G \in \call F(\scr A)$ with $F \leq G$ and identify 
  \begin{displaymath}
    \scr A[F] = \scr A_F = \{ H \in \scr A\,|\, F \subseteq H\},
  \end{displaymath}
  in particular we have an inclusion $\scr A[G] \subseteq \scr A[F]$
  and, correspondingly, a function
  $i_{F \leq G}: \call F(\scr A[G]) \to \call F(\scr A[F])$ as in
  Definition \ref{def:Fm}, which induces a function $j_{F\leq G}:
\Sal(\scr A[F]) \to \Sal (\scr A[G])$ as in Definition \ref{def:DDD}.
\end{definition}


\begin{teo}[Lemma 3.2.8 and Theorem 4.2.1 of \cite{delucchi_phd}]
  The assignment $\scr E:\call F(\scr A)  \to \acicat^{op}$, $\scr
  E(F):=\Sal(\scr A[F])$, $\scr E(F\leq G):=j_{F\leq G}$ defines a
  diagram of posets such that $\colim \scr E$ is poset isomorphic to $ \Sal (\scr A)$.
\end{teo}

The stratification of $\Sal (\scr A)$ is also defined along the lines
of the preceding sections.

\begin{definition}
  Define the map $\theta: \Sal(\scr A) \to \scr Y$ as follows
  \begin{displaymath}
    \theta([F, C]) = (X(F, i_{F \leq G}(G)), \sigma_{\scr A[X(F, i_{F \leq G}(G))]}(G)).
  \end{displaymath}
  where we identified $G = \min \call L(\scr A[G])$.
\end{definition}

\begin{definition}
  Let $\scr A$ be a finite complexified affine hyperplane arrangement and define a
  total ordering $\dashv$ on $\scr Y$ as in Definition \ref{def:total}. Define:
  \begin{displaymath}
    \call S_{(Y,C)} = \bigg\{ [F,C] \in \Sal(\scr A) \bigg\vert
    \begin{array}{c}  
     \textrm{there is } [G,K] \in \Sal(\scr A) \textrm{ with }\\
       {[F,C]} \leq [G,K] \textrm{ and }\theta([G,K]) = (Y,C) 
     \end{array}
     \bigg\}
  \end{displaymath}
  \begin{displaymath}
    \call N_{(Y,C)} = S_{(Y,C)} \backslash \bigcup_{(Y', C') \dashv (Y,C)} \call S_{(Y',C')}.
  \end{displaymath}
\end{definition}

The arguments of Section \ref{sec:topstrata} can now be adapted to the affine case,
obtaining the following analogon of Theorem \ref{teo:pezzinuovi}.
\begin{teo}
  Let $\scr A$ be a finite complexified affine hyperplane arrangement. There is an isomorphism
  of posets
  \begin{displaymath}
    \call N_{(Y,C)} \cong \call F(\scr A^Y)^{op}\quad\quad
    \textrm{ for all } (Y,C) \in \scr Y.
  \end{displaymath}
\end{teo}

The analogon of Proposition \ref{prop.minreale} is proved in
\cite[Theorem 4.5.7 and Corollary 4.5.8]{BLSWZ},
from which it follows that the poset $\call N_{(Y,C)}^{op}$ is shellable, and therefore 
$\call N_{(Y,C)}$ admits an acylic matching with one critical cell in top dimension.
Applying the Patchwork Lemma as in Proposition \ref{prop:mintorico} we
obtain a perfect acyclic matching $\mathfrak M$ of $\Sal(\scr A)$. 
We summarize.
\begin{prop}
  Let $\scr A$ be a finite complexified affine hyperplane
  arrangement. The (affine) oriented matroid data of $\scr A$ intrinsecally define a discrete
  Morse function on $\Sal(\scr A)$ that collapses the Salvetti complex
  to a minimal complex. 
\end{prop}

\begin{oss}
  The considerations of this section carry over to the general case of
  nonstretchable affine oriented matroids, as in \cite{delucchi} for
  the non-affine case.
\end{oss}

\bibliographystyle{plain}
\bibliography{toric}

\def\cprime{$'$}
\begin{thebibliography}{10}

\bibitem{Benedetti}
B.~{Benedetti}.
\newblock {Discrete Morse Theory for Manifolds with Boundary}.
\newblock {\em ArXiv e-prints (to appear in Transactions of the American
  Mathematical Society)}, July 2010.

\bibitem{BLSWZ}
Anders Bj{\"o}rner, Michel Las~Vergnas, Bernd Sturmfels, Neil White, and
  G{\"u}nter~M. Ziegler.
\newblock {\em Oriented matroids}, volume~46 of {\em Encyclopedia of
  Mathematics and its Applications}.
\newblock Cambridge University Press, Cambridge, second edition, 1999.

\bibitem{brieskorn1971gt}
E.~Brieskorn.
\newblock {Sur les groupes de tresses}.
\newblock {\em Seminaire Bourbaki}, 72:21--44, 1971.

\bibitem{curtis}
Charles~W. Curtis and Irving Reiner.
\newblock {\em Representation theory of finite groups and associative
  algebras}.
\newblock AMS Chelsea Publishing, Providence, RI, 2006.
\newblock Reprint of the 1962 original.

\bibitem{mocidadd}
M.~{D'Adderio} and L.~{Moci}.
\newblock {Arithmetic matroids, Tutte polynomial, and toric arrangements}.
\newblock {\em ArXiv e-prints}, May 2011.

\bibitem{dantoniodelucchi}
Giacomo d'Antonio and Emanuele Delucchi.
\newblock A salvetti complex for toric arrangements and its fundamental group.
\newblock {\em International Mathematics Research Notices}, 2011.

\bibitem{davis_settepanella}
M.~W. {Davis} and S.~{Settepanella}.
\newblock {Vanishing results for the cohomology of complex toric hyperplane
  complements}.
\newblock {\em ArXiv e-prints}, November 2011.

\bibitem{de2005geometry}
C.~De~Concini and C.~Procesi.
\newblock {On the geometry of toric arrangements}.
\newblock {\em Transformation Groups}, 10(3):387--422, 2005.

\bibitem{deconcini2010topics}
C.~De~Concini and C.~Procesi.
\newblock {\em {Topics in hyperplane arrangements, polytopes and box-splines}}.
\newblock Springer Verlag, 2010.

\bibitem{DPV}
C.~De~Concini, C.~Procesi, and M.~Vergne.
\newblock Vector partition functions and index of transversally elliptic
  operators.
\newblock {\em Transform. Groups}, 15(4):775--811, 2010.

\bibitem{DeluSette}
E.~Delucchi and S.~Settepanella.
\newblock Combinatorial polar orderings and recursively orderable arrangements.
\newblock {\em Adv. in Appl. Math.}, 44(2):124--144, 2010.

\bibitem{delucchi_phd}
Emanuele Delucchi.
\newblock {\em Topology and combinatorics of arrangement covers and of nested
  set complexes.}
\newblock PhD thesis, ETH Z\"urich, Summer 2006.

\bibitem{delucchi}
Emanuele Delucchi.
\newblock Shelling-type orderings of regular {CW}-complexes and acyclic
  matchings of the {S}alvetti complex.
\newblock {\em Int. Math. Res. Not. IMRN}, (6):Art. ID rnm167, 39, 2008.

\bibitem{MR2018927}
Alexandru Dimca and Stefan Papadima.
\newblock Hypersurface complements, {M}ilnor fibers and higher homotopy groups
  of arrangments.
\newblock {\em Ann. of Math. (2)}, 158(2):473--507, 2003.

\bibitem{ehrenborg2009affine}
R.~Ehrenborg, M.~Readdy, and M.~Slone.
\newblock {Affine and toric hyperplane arrangements}.
\newblock {\em Discrete and Computational Geometry}, 41(4):481--512, 2009.

\bibitem{GMS}
G.~Gaiffi, F.~Mori, and M.~Salvetti.
\newblock Minimal {CW}-complexes for complements to line arrangements via
  discrete {M}orse theory.
\newblock In {\em Topology of algebraic varieties and singularities}, volume
  538 of {\em Contemp. Math.}, pages 293--308. Amer. Math. Soc., Providence,
  RI, 2011.

\bibitem{gaiffisalvetti}
Giovanni Gaiffi and Mario Salvetti.
\newblock The {M}orse complex of a line arrangement.
\newblock {\em J. Algebra}, 321(1):316--337, 2009.

\bibitem{Hi93}
Eriko Hironaka.
\newblock Abelian coverings of the complex projective plane branched along
  configurations of real lines.
\newblock {\em Mem. Amer. Math. Soc.}, 105(502):vi+85, 1993.

\bibitem{jambuterao}
Michel Jambu and Hiroaki Terao.
\newblock Arrangements of hyperplanes and broken circuits.
\newblock In {\em Singularities ({I}owa {C}ity, {IA}, 1986)}, volume~90 of {\em
  Contemp. Math.}, pages 147--162. Amer. Math. Soc., Providence, RI, 1989.

\bibitem{kozlov2007combinatorial}
Dmitry Kozlov.
\newblock {\em Combinatorial algebraic topology}, volume~21 of {\em Algorithms
  and Computation in Mathematics}.
\newblock Springer, Berlin, 2008.

\bibitem{looij}
Eduard Looijenga.
\newblock Cohomology of {${\scr M}_3$} and {${\scr M}^1_3$}.
\newblock In {\em Mapping class groups and moduli spaces of {R}iemann surfaces
  ({G}\"ottingen, 1991/{S}eattle, {WA}, 1991)}, volume 150 of {\em Contemp.
  Math.}, pages 205--228. Amer. Math. Soc., Providence, RI, 1993.

\bibitem{mocicombinatorics}
L.~{Moci}.
\newblock {Combinatorics and topology of toric arrangements defined by root
  systems}.
\newblock {\em Rend. Lincei Mat. Appl.}, 19(4):293--308, 2008.

\bibitem{mocitutte2009}
L.~Moci.
\newblock A tutte polynomial for toric arrangements.
\newblock {\em Trans. Amer. Math. Soc}, 364:1067--1088, 2012.

\bibitem{mocisettepanella}
L.~Moci and S.~Settepanella.
\newblock The homotopy type of toric arrangements.
\newblock {\em Journal of Pure and Applied Algebra}, 215(8):1980 -- 1989, 2011.

\bibitem{orlik1980cat}
P.~Orlik and L.~Solomon.
\newblock {Combinatorics and topology of complements of hyperplanes}.
\newblock {\em Inventiones Mathematicae}, 56(1):167--189, 1980.

\bibitem{orlik1992ah}
Peter Orlik and Hiroaki Terao.
\newblock {\em Arrangements of hyperplanes}, volume 300 of {\em Grundlehren der
  Mathematischen Wissenschaften [Fundamental Principles of Mathematical
  Sciences]}.
\newblock Springer-Verlag, Berlin, 1992.

\bibitem{MR1900880}
Richard Randell.
\newblock Morse theory, {M}ilnor fibers and minimality of hyperplane
  arrangements.
\newblock {\em Proc. Amer. Math. Soc.}, 130(9):2737--2743 (electronic), 2002.

\bibitem{salvetti1987tcr}
M.~Salvetti.
\newblock {Topology of the complement of real hyperplanes in $\mat C^N$}.
\newblock {\em Inventiones mathematicae}, 88(3):603--618, 1987.

\bibitem{Sa88}
Mario Salvetti.
\newblock Arrangements of lines and monodromy of plane curves.
\newblock {\em Compositio Math.}, 68(1):103--122, 1988.

\bibitem{Sa882}
Mario Salvetti.
\newblock On the homotopy type of the complement to an arrangement of lines in
  {${\bf C}^2$}.
\newblock {\em Boll. Un. Mat. Ital. A (7)}, 2(3):337--344, 1988.

\bibitem{MR2350466}
Mario Salvetti and Simona Settepanella.
\newblock Combinatorial {M}orse theory and minimality of hyperplane
  arrangements.
\newblock {\em Geom. Topol.}, 11:1733--1766, 2007.

\bibitem{yoshinaga2007}
Masahiko Yoshinaga.
\newblock Hyperplane arrangements and {L}efschetz's hyperplane section theorem.
\newblock {\em Kodai Math. J.}, 30(2):157--194, 2007.

\bibitem{zaslavsky1975facing}
T.~Zaslavsky.
\newblock {Facing up to arrangements: face-count formulas for partitions of
  space by hyperplanes}.
\newblock {\em Mem. Amer. Math. Soc}, 1(1):154, 1975.

\end{thebibliography}

\vfill

\noindent{\em Giacomo d'Antonio,} Fachbereich Mathematik und Informatik, Universit\"at Bremen,
Bibliothekstra\ss{}e 1, 28359 Bremen, Germany.\\

\noindent{\em Emanuele Delucchi,} Fachbereich Mathematik und Informatik, Universit\"at Bremen,
Bibliothekstra\ss{}e 1, 28359 Bremen, Germany.
\end{document}